\documentclass[reqno]{amsart}

\usepackage{amsrefs}
\usepackage{amssymb}
\usepackage{hyperref}
\usepackage{fancyhdr}

\newtheorem{theorem}{Theorem}
\newtheorem*{theorem*}{Theorem}
\newtheorem{lemma}[theorem]{Lemma}
\newtheorem{proposition}[theorem]{Proposition}
\newtheorem{claim}[theorem]{Claim}
\newtheorem{corollary}[theorem]{Corollary}

\newtheorem{maintheorem}{Theorem}

\theoremstyle{exercise}

\theoremstyle{definition}

\newtheorem*{definition*}{Definition}
\newtheorem*{lemma*}{Lemma}
\usepackage{graphicx}
\usepackage{pstricks, enumerate, pst-node, pst-text, pst-plot}

\numberwithin{equation}{section}
\numberwithin{theorem}{section}

\newcommand{\intav}[1]{\mathchoice {\mathop{\vrule width 6pt height 3 pt depth  -2.5pt
\kern -8pt \intop}\nolimits_{\kern -6pt#1}} {\mathop{\vrule width
5pt height 3  pt depth -2.6pt \kern -6pt \intop}\nolimits_{#1}}
{\mathop{\vrule width 5pt height 3 pt depth -2.6pt \kern -6pt
\intop}\nolimits_{#1}} {\mathop{\vrule width 5pt height 3 pt depth
-2.6pt \kern -6pt \intop}\nolimits_{#1}}}

\newcommand{\intavl}[1]{\mathchoice {\mathop{\vrule width 6pt height 3 pt depth  -2.5pt
\kern -8pt \intop}\limits_{\kern -6pt#1}} {\mathop{\vrule width 5pt
height 3  pt depth -2.6pt \kern -6pt \intop}\nolimits_{#1}}
{\mathop{\vrule width 5pt height 3 pt depth -2.6pt \kern -6pt
\intop}\nolimits_{#1}} {\mathop{\vrule width 5pt height 3 pt depth
-2.6pt \kern -6pt \intop}\nolimits_{#1}}}

\newcommand{\R}{\mathbb{R}}
\newcommand{\N}{\mathbb{N}}

\newcommand{\Z}{\mathbb{Z}}
\newcommand{\eps}{\varepsilon}

\renewcommand{\P}[1]{{\mathbb{P}}\left[{#1}\right]}
\newcommand{\CondP}[2]{{\mathbb{P}}\left[{#1}\middle\vert{#2}\right]}
\newcommand{\CondE}[2]{{\mathbb{E}}\left[{#1}\middle\vert{#2}\right]}

\newcommand{\Esub}[2]{{\mathbb{E}}_{#1}\left[{#2}\right]}

\newcommand{\E}[1]{{\mathbb{E}}\left[{#1}\right]}

\newcommand{\symdiff}{\triangle}

\newcommand{\bnd}{\mathrm{{\bf bnd}}}
\newcommand{\pb}{\Pi}
\newcommand{\sub}[1]{\mathrm{Sub}_{#1}}
\newcommand{\lampbnd}{{\mathrm{\bf conf}}}
\newcommand{\finconf}{C_C}
\newcommand{\lamps}{{\mathrm{conf}}}
\newcommand{\cA}{{\mathcal{A}}}
\newcommand{\Ent}[1]{H\left({#1}\right)}

\newcommand{\mup}{\eta}
\newcommand{\half}{{\textstyle \frac12}}

\newcommand{\mS}{\mathcal{S}}

\newcommand{\irs}{\mathcal{M}}

\DeclareMathOperator{\supp}{supp}

\begin{document}

\title[]{Furstenberg entropy realizations for virtually free groups
  and lamplighter groups}

\author[Yair Hartman]{Yair Hartman}
\address{Weizmann Institute of Science, Faculty of Mathematics and Computer Science, POB 26, 76100, Rehovot, Israel.}
\email{yair.hartman, omer.tamuz@weizmann.ac.il}

%    Information for third author
\author[Omer Tamuz]{Omer Tamuz}

\thanks{Y.\ Hartman is supported by the European Research Council,
  grant 239885.  O.\ Tamuz is supported by ISF grant 1300/08, and is a
  recipient of the Google Europe Fellowship in Social Computing. This
  research is supported in part by this Google Fellowship.}
\date{\today}

\begin{abstract}
  Let $(G,\mu)$ be a discrete group with a generating probability
  measure. Nevo shows that if $G$ has property (T) then there exists
  an $\eps>0$ such that the Furstenberg entropy of any
  $(G,\mu)$-stationary ergodic space is either zero or larger than
  $\eps$.

  Virtually free groups, such as $SL_2(\Z)$, do not have property (T),
  and neither do their extensions, such as surface groups. For these,
  we construct stationary actions with arbitrarily small, positive
  entropy. This construction involves building and lifting spaces of
  lamplighter groups. For some classical lamplighters, these spaces
  realize a dense set of entropies between zero and the Poisson
  boundary entropy.
\end{abstract}

\maketitle
\tableofcontents
\section{Introduction}
Let $G$ be a countable discrete group, and let $\mu$ be a generating probability
measure. A $G$-space $X$ with a probability measure $\nu$ is called a
$(G,\mu)$-stationary space if $\sum_g\mu(g)g\nu = \nu$. Hence a
stationary measure $\nu$ is not in general $G$-invariant, but it is
invariant ``on average'', when the average is taken over $\mu$. An
important invariant of stationary spaces is the Furstenberg
entropy~\cite{furstenberg1963noncommuting}, given by
\begin{align*}
  h_\mu(X,\nu) = \sum_{g \in G}\mu(g)\int_X-\log\frac{d\nu}{dg\nu}(x)dg\nu(x).
\end{align*}

Despite the fact that stationary spaces have been studied for several
decades now, few examples are known, and the theory of their structure
and properties is still far from
complete~\cite{furstenberg2009stationary}. For example, it is in
general not known which Furstenberg entropy values they may take; this
problem is called the {\em Furstenberg entropy realization
  problem}~\cite{nevo2000rigidity, bowen2010random}. More
specifically, it is not known which groups have an entropy gap:
\begin{definition*}
  $(G,\mu)$ has an {\em entropy gap} if it admits stationary spaces of
  positive Furstenberg entropy, and if there exists an $\eps>0$ such
  that the Furstenberg entropy of any ergodic $(G,\mu)$-stationary
  space is either zero or greater than $\eps$.

  A group $G$ has an {\em entropy gap} if $(G,\mu)$ has an entropy gap
  for every generating measure $\mu$ with finite entropy.
\end{definition*}

Nevo~\cite{nevo2003spectral} shows that any group with Kazhdan's
property (T) has an entropy gap.  We show that a large class of
discrete groups without property (T) do not have an entropy gap.
\begin{maintheorem}
  \label{thm:main-realization}
  Let $G$ be a finitely generated virtually free group. Let $\mu$ be a
  generating measure on $G$, with finite first moment. Then $(G,\mu)$
  does not have an entropy gap.

  In particular, any finitely generated virtually free group does not
  have an entropy gap.
\end{maintheorem}
A virtually free group is a group that has a free group as a finite
index subgroup. In particular, $SL_2(\Z)$ is virtually free, and so
does not have an entropy gap.

L.\ Bowen~\cite{bowen2010random} introduces a new example of
stationary spaces based on invariant random subgroups (see also
earlier work by Kaimanovich~\cite{kaimanovich2005amenability}). We
shall refer to these spaces as {\em Bowen spaces}.  He uses some
insights into their entropy to show that, for free groups with the
uniform measure over the generators, any entropy between zero and the
Poisson boundary entropy can be realized. We also realize entropies
using Bowen spaces: to prove Theorem~\ref{thm:main-realization}, we
construct Bowen spaces of lamplighter groups, and lift them to Bowen
spaces of virtually free groups. Using a recent result of Hartman,
Lima and Tamuz~\cite{hartman2012abramov} (see
also~\cite{kaimanovich1992discretization}) that relates the entropies
of the actions of groups and their finite index subgroups, we control
the entropies of the lifted spaces, and show that they can be made
arbitrarily small.

A natural stationary space of lamplighter groups is the limit
configuration boundary (see Section~\ref{sec:lamp-config}), which, in
some classical lamplighters, has been shown to coincide with the
Poisson boundary (for example, this is known for $(\Z/2\Z \wr
\Z^d,\mu)$ for $d\geq 5$ and $\mu$ with finite third moment;
see~\cite{erschler2011poisson,kaimanovich2000poisson,karlsson2007poisson}).
We denote by $h_\lamps(G,\mu)$ the Furstenberg entropy of the
configuration boundary.  Our construction of Bowen spaces for
lamplighter groups yields the following realization result.
\begin{maintheorem}
  \label{thm:lamplighter-realization}
  Let $G=L \wr \Gamma$ be a finitely generated discrete lamplighter
  with base group $\Gamma$ and lamps in $L$. Let $\mu$ be a generating
  measure on $G$ with finite entropy, and such that its projected
  random walk on $\Gamma$ has a trivial Poisson boundary. Then there
  exists a dense set $H \subseteq [0,h_\lamps(G,\mu)]$ such that for
  each $h \in H$ there exists an ergodic $(G,\mu)$-stationary space
  with Furstenberg entropy $h$.
\end{maintheorem}

The property of not having an entropy gap is closed under group
extensions in the following sense. If $G\xrightarrow{\varphi}Q$ is a
surjective group homomorphism, and if $(Q,\varphi_*\mu)$ does not have
an entropy gap for some generating measure $\mu$ on $G$, then
$(G,\mu)$ does not have an entropy gap. Indeed, if $(X,\nu)$ is a
$(Q,\varphi_*\mu)$-stationary space, then it is also
$(G,\mu)$-stationary for the $G$ action factored through $\varphi$,
and furthermore
\begin{align*}
  h_{\mu}(X,\nu) = h_{\varphi_*\mu}(X,\nu).
\end{align*}
The following result is therefore a consequence of
Theorems~\ref{thm:main-realization}
and~\ref{thm:lamplighter-realization}.
\begin{maintheorem}
  Let $(G,\mu)$ be a discrete group with a generating measure such
  that there exists a surjective group homomorphism
  $G\xrightarrow{\varphi}Q$, where $(Q,\varphi_*\mu)$ is either
  virtually free with finite first moment, or a lamplighter satisfying
  the conditions of Theorem~\ref{thm:lamplighter-realization} and such
  that $h_\lamps(Q,\varphi_*,\mu)>0$. Then $(G,\mu)$ has no entropy
  gap.
\end{maintheorem}
In particular, this holds for surface groups with finite
first moment measures, as these are extensions of free groups, and
since finite first moment measures are pushed forward to finite first
moment measures (Claim~\ref{lem:quotient-preserves-ffm}). 

\subsection{Related results}
We conclude this introduction with a short survey of previous work on
Furstenberg entropy realization.

The Furstenberg entropy of any stationary space is bounded from above
by the entropy of the Poisson
boundary~\cite{furstenberg1971random}. Furthermore, Kaimanovich and
Vershik~\cite{kaimanovich1983random} show that when $\Ent{\mu}$, the
entropy of $\mu$, is finite, then the Furstenberg entropy of the
Poisson boundary is equal to $h_{RW}(G,\mu)$, the random walk entropy
of $\mu$, defined by
\begin{align*}
  h_{RW}(G,\mu) =  \lim_{n \to \infty}\frac{1}{n}\Ent{\mu^n}.
\end{align*}

Little is known about which entropy values between $0$ and
$h_{RW}(G,\mu)$ can be realized by ergodic stationary spaces. Note
that any entropy in this range can be realized with a non-ergodic
space that is a convex linear combination of the Poisson boundary and
a trivial space.

When $G$ is abelian, or more generally, virtually nilpotent, then the
entropy of the Poisson boundary, and hence of any stationary space,
vanishes for any $\mu$~\cite{kaimanovich1983random}. Furstenberg shows
that a group $G$ is amenable if and only if there exists a generating
measure $\mu$ such that the entropy of the Poisson boundary vanishes.

Nevo and Zimmer~\cite{nevo2000rigidity} show that $PSL_2(\R)$, and
more generally, any simple Lie group with $\R$-rank $\geq 2$ with a
parabolic subgroup that maps onto $PSL_2(\R)$, has infinitely many
distinct realizable entropy values, for any admissible measure. No
other guarantees are given regarding these values.

Finally, as we mentioned above, Nevo~\cite{nevo2003spectral} shows
that when $G$ has Kazhdan's property (T) then it has an entropy gap,
and Bowen~\cite{bowen2010random} shows that any entropy between $0$
and $h_{RW}(G,\mu)$ can be realized when $G$ is a free group of rank
$2 \leq n < \infty$ and $\mu$ is the uniform measure on its
generators.

\subsection{Acknowledgments}
We are grateful to Yuri Lima for many useful discussions. We would
also like to thank Uri Bader and Amos Nevo for motivating
conversations, to thank Lewis Bowen and Vadim Kaimanovich for
commenting on the first draft of this paper, and to thank the referee
for many insightful comments and suggestions.

\vspace{12 pt}

The remainder of the paper is organized as follows. In
Section~\ref{sec:preliminaries} we give general definitions and
notation, and in particular elaborate on Bowen spaces. In addition, we
show that, for the purpose of entropy realization, it is possible to
assume without loss of generality that $\mu$ is supported
everywhere. In Section~\ref{sec:lamplighter} we prove our realization
result for lamplighters, Theorem~\ref{thm:lamplighter-realization},
and in Section~\ref{sec:lifting} we prove our main result,
Theorem~\ref{thm:main-realization}.

\section{Preliminaries and notation}
\label{sec:preliminaries}
\subsection{Random walks on groups}
Let $G$ be a discrete group, and let $\mathcal{P}(G)$ be the space of
probability measures on $G$. $\mu \in \mathcal{P}(G)$ is a {\em
  generating measure} if $G$ is the semigroup generated by its
support. We assume henceforth that $\mu$ is generating.

Consider the case that $G$ is finitely generated, and let
$\mathcal{S}$ be a finite symmetric generating set of $G$. Define the
{\em word length metric} of $G$ w.r.t.\ $\mS$ to be
$|g|_\mS=\min\left\{ n|s_{1}\cdots
  s_{n}=g,s_{i}\in\mS\right\} $.  The measure $\mu$ has {\em
  finite first moment} if
\begin{align*}
  \sum_{g\in G}\mu(g)|g|_\mS<\infty.  
\end{align*}
Since word length metrics induced by different finite generating sets are
bi-lipschitz equivalent, the property of having finite first moment does
not depend on the choice of $\mS$.

If $\mu$ has finite first moment then it has finite entropy
$\Ent{\mu}$, given by
\begin{align*}
  \Ent{\mu} = \sum_{g \in G}-\mu(g)\log\mu(g).
\end{align*}

For $n \in \N$, let $X_n$ be i.i.d.\ random variables taking values in
$G$, with law $\mu$, and let $Z_n = X_1\cdots X_n$.  A $\mu$ random
walk on $G$ is a measure $\mathbb{P}$ on $\Omega =
G^\N$, such that $(Z_1,Z_2,\ldots) \sim \mathbb{P}$.

\subsection{Stationary spaces and the Poisson boundary}
\label{sec:stationary}
Let $(X,\nu)$ be a Lebesgue probability space, equipped with a
measurable $G$-action $G \times X \xrightarrow{a} X$. We denote by
$g\nu$ the measure defined by $(g\nu)(E) = \nu\left(g^{-1}E\right)$.
$(X,\nu)$ is a $(G,\mu)$-{\em stationary space} if
\begin{align*}
  \mu * \nu = \sum_{g \in G}\mu(g)g\nu = \nu,
\end{align*}
where $\mu * \nu$, the convolution of $\mu$ with $\nu$, is the image
of $\mu \times \nu$ under the action $a$. It follows from stationarity
and the fact that $\mu$ is generating that $\nu$ and $g\nu$ are
mutually absolutely continuous for all $g \in G$.

Let $(X,\nu)$ be a $(G,\mu)$-stationary space, and let $(Y,\eta)$ be a
$G$-space. A measurable map $\pi : X \to Y$ is a $G$-{\em factor} if it is
$G$-equivariant (i.e., if $\pi$ commutes with the $G$-actions) and if
$\pi_*\nu = \eta$. In this case $(Y,\eta)$ is called a $G$-{\em
  factor} of $(X,\nu)$, and it follows that $(Y,\eta)$ is also
$(G,\mu)$-stationary. If, in addition, $\pi$ is an isomorphism of the
probability spaces, then $\pi$ is a $G$-{\em isomorphism}.

An important $(G,\mu)$-stationary space is $\pb(G,\mu)$, the {\em
  Poisson boundary} of $(G,\mu)$. The Poisson boundary can be defined
as the Mackey realization~\cite{mackey1962point} of the shift
invariant sigma-algebra of the space of random walks
$(\Omega,\mathbb{P})$~\cite{zimmer1978amenable,kaimanovich1983random},
also known as the space of shift ergodic components of
$\Omega$. Furstenberg's original
definition~\cite{furstenberg1971random} used the Gelfand
representation of the algebra of bounded $\mu$-harmonic functions on
$G$.  For formal definitions see also Furstenberg and
Glasner~\cite{furstenberg2009stationary}, or a survey by
Furman~\cite{furman2002random}.

$G$-factors of the Poisson boundary are stationary spaces called
$(G,\mu)$-{\em boundaries}; the Mackey realization of each
$G$-invariant, shift invariant sigma-algebra is a $(G,\mu)$-boundary.
A different perspective is that a compact $G$-space $(X,\nu)$ is a
$(G,\mu)$-boundary if it is a $(G,\mu)$-stationary space such that
$\lim_nZ_n\nu$, in the weak* topology, is almost surely a point mass
measure $\delta_x\in\mathcal{P}(X)$. The map $\bnd_X :\Omega \to X$
that assigns to $(Z_1,Z_2,\ldots)$ the point $x$ is called the {\em
  boundary map} of $(X,\nu)$.  For further discussion and a definition
of boundaries that is independent of topology see Bader and
Shalom~\cite{bader2006factor}.

We shall also consider the Poisson boundary of a general Markov chain,
defined again as the space of ergodic components of the shift
invariant sigma-algebra~\cite{kaimanovich1992measure}.

%$\pb(G,\mu)$ is a $(G,\mu)$-stationary space, and the Furstenberg
%entropy of any other stationary space is bounded from above by its
%entropy~\cite{furstenberg1971random,kaimanovich1983random}.
%
\subsection{Furstenberg entropy}

The {\em Furstenberg entropy} of a $(G,\mu)$-stationary space
$(X,\nu)$ is given by
\begin{align*}
  h_\mu(X,\nu) = \sum_{g \in G}\mu(g)\int_X-\log\frac{d\nu}{dg\nu}(x)dg\nu(x).
\end{align*}
Alternatively, it can be written as
\begin{align*}
  h_\mu(X,\nu) = \E{D_{KL}(Z_1\nu||\nu)},
\end{align*}
where $D_{KL}$ denotes the {\em Kullback-Leibler divergence}. Since
the latter decreases under factors, it follows that if $(Y,\eta)$ is a
$G$-factor of $(X,\nu)$, then $h_\mu(X,\nu) \geq h_\mu(Y,\eta)$.

A space $(X,\nu)$ is $G$-invariant if and only if $h_\mu(X,\nu)=0$,
and, in general, the Furstenberg entropy can be thought of as
quantifying the $\mu$-average deformation of $\nu$ by the $G$-action.

\subsection{The induced walk on finite index subgroups}
\label{sec:subgroups}
Let $\Gamma$ be a finite index subgroup of $G$, and let $\tau =
\min_n\{Z_n \in \Gamma\}$ be the $\Gamma$ hitting time of the $\mu$
random walk. $\tau$ is almost surely finite, and so it is possible to
define the hitting measure $\theta \in \mathcal{P}(\Gamma)$ as the law
of $Z_\tau$. The $\theta$ random walk on $\Gamma$ is intimately
related to the $\mu$ random walk on $G$. In particular, any
$(G,\mu)$-stationary space is also a $(\Gamma,\theta)$-stationary
space. Furthermore, Furstenberg~\cite{furstenberg1971random} shows
that the Poisson boundaries of the two walks are identical.

It is shown in~\cite{hartman2012abramov} that $\E{\tau} = [G:\Gamma]$,
and that for any $(G,\mu)$-stationary space $(X,\nu)$ it holds that
\begin{align}
  \label{eq:abramov}
  h_\theta(X,\nu) = [G:\Gamma] \cdot h_\mu(X,\nu).
\end{align}

\subsection{Bowen spaces}
\label{sec:bowen-spaces}
In~\cite{bowen2010random}, Bowen introduces a novel example of stationary
spaces, which we refer to as Bowen spaces. As we make extensive use of
these spaces, we would like to motivate their definition and elaborate
on it.

Let $(B,\nu)=\pb(G,\mu)$ be the Poisson boundary of $(G,\mu)$, let
$(Z_1,Z_2,\ldots)$ be a $\mu$ random walk on $G$, and let $K$ be a
normal subgroup of $G$ with $G\xrightarrow{\varphi} K \backslash
G$. Then $(KZ_1,KZ_2,\ldots)$ is a $\varphi_*\mu$ random walk on the
group $K \backslash G$, which we call the induced random walk. $\pb(K
\backslash G,\varphi_*\mu)$, the Poisson boundary of $(K \backslash
G,\varphi_*\mu)$, is a factor of $\pb(G,\mu)$; the former is
isomorphic to the space of ergodic components of the $K$ action on the
latter. $\pb(K \backslash G,\varphi_*\mu)$ is therefore also a
$G$-space, and, furthermore, a
$(G,\mu)$-boundary~\cite{bader2006factor}.

When $K$ is not normal, we can still consider the induced Markov chain
$(KZ_1, KZ_2, \ldots)$, which is, however, no longer a random walk on a
group. The action of $g \in G$ on the $\mu$ random walk descends to
\begin{align}
  \label{eq:mc-g-action}
  g(KZ_1,KZ_2,\ldots) = (gKZ_1,gKZ_2,\ldots) = (K^ggZ_1,K^ggZ_2,\ldots),
\end{align}
which maps the Markov chain on $K \backslash G$ starting from $K$ to
the chain on $K^g \backslash G$ starting from $K^gg$, where
$K^g=gKg^{-1}$. Denote by $P_K^n(Kg,Kh)$ the transition probability
from $Kg$ to $Kh$ in $n$ steps of the induced chain.

Even though the induced Markov chain on $K \backslash G$ is not a
random walk on a group, we can still consider its Poisson boundary,
$(B_K,\nu_K)$, and a boundary map $(K \backslash G)^\N \to B_K$. As in
the normal case, it is a factor of $\pb(G,\mu)$. However, in this case
$(B_K,\nu_K)$ does not admit a natural $G$-action; the induced action
of $g \in G$ on $(B_K,\nu_K)$ maps it to $(B_{K^g},\nu_{K^gg})$, where
$\nu_{K^gg}$ is the measure on the Poisson boundary of the Markov
chain $K^g \backslash G$ that starts at $K^gg$.

To build a $G$-space, Bowen considers a larger space, namely that of
all Poisson boundaries of the form $B_K$.  Denote by $\sub{G}$ the
space of all subgroups of $G$ equipped with the topology of
convergence on finite subsets, and denote
\begin{align*}
  B(\sub{G}) = \{(K,x)\,:\,K \in \sub{G}, x \in B_K\}.
\end{align*}
This space can be thought of as the Mackey realization of
$\sub{G}\times G^\N$ with the Borel sub-sigma-algebra generated by the
shift $(K,g_1,g_2,\ldots) \mapsto (K,g_2,g_3,\ldots)$ and the quotient
$(K,g_1,g_2,\ldots) \mapsto (K,Kg_1,Kg_2,\ldots)$. As such it is
equipped with the derived Borel structure.

The $G$-action of Eq.~\ref{eq:mc-g-action} on Markov chains descends,
via composition with the boundary map, to a $G$-action on
$B(\sub{G})$.

To construct a stationary measure over $B(\sub{G})$, let $\lambda \in
\mathcal{P}(\sub{G})$ be an invariant random subgroup (IRS) measure -
a measure on $\sub{G}$ that is invariant to conjugation. Let
$\nu_\lambda \in \mathcal{P}(B(\sub{G}))$ be given by
$d\nu_\lambda(K,x) = d\nu_K(x) d\lambda(K)$. This is the measure that
gives the fiber above $K$ the measure $\nu_K$, with measure $\lambda$
over the fibers.

Bowen shows that $(B(\sub{G}), \nu_\lambda)$ is $(G,\mu)$-stationary,
and is furthermore ergodic if $\lambda$ is ergodic. We refer to this
space as the {\em Bowen space} associated with $\lambda$.

By definition, the Furstenberg entropy of a Bowen space is given by
\begin{align*}
  h_\mu(B(\sub{G}),\nu_\lambda) = \sum_{g \in
    G}\mu(g)\int_{B(\sub{G})}-\log\frac{d\nu_\lambda}{dg\nu_\lambda}(K,x)
  dg\nu_\lambda(K,x).
\end{align*}
Using $d\nu_\lambda(K,x) = d\nu_K(x)  d\lambda(K)$ and the fact
that $g\lambda = \lambda$ and $g\nu_K=\nu_{K^gg}$, 
\begin{align*}
   = \sum_{g \in
    G}\mu(g)\int_{\sub{G}}\int_{B_K}-\log\frac{d\nu_K}{d\nu_{Kg}}(x)d\nu_{Kg}(x)
  d\lambda(K).
\end{align*}
Even though $(B_K,\nu_K)$ is not a $(G, \mu)$-stationary space - in
fact, not even a $G$-space - it will help us to define its $(G,\mu)$
Furstenberg entropy by
\begin{align}
  \label{eq:fiber-ent}
  h_\mu(B_K,\nu_K) = \sum_{g \in
    G}\mu(g)\int_{B_K}-\log\frac{d\nu_K}{d\nu_{Kg}}(x)d\nu_{Kg}(x),
\end{align}
so that
\begin{align}
  \label{eq:bowen-space-entropy}
  h_\mu(B(\sub{G}),\nu_\lambda) = \int_{\sub{G}}h_\mu(B_K,\nu_K) d\lambda(K).
\end{align}

An alternative way to understand Eq.~\ref{eq:bowen-space-entropy} is
to regard $(\sub{G},\lambda)$ as a $G$-factor of
$(B(\sub{G}),\nu_\lambda)$. In general, if $(Y,\lambda)$ is a
$G$-factor of $(X,\nu)$, then it is possible to express the entropy of
$X$ as a sum of the entropy of $Y$ and the average entropy of the
fibers $X_y = \pi^{-1}(y)$:
\begin{align}
  \label{eq:rochlin1}
  h_\mu(X,\nu) = h_\mu(Y,\lambda) + \int_Yh_\mu(X_y, \nu_y)d\lambda(y),
\end{align}
where
\begin{align}
  \label{eq:rochlin2} 
 h_\mu(X_y, \nu_y) = \sum_{g \in G}\mu(g)\int_X-\log\frac{d\nu_{gy}}{dg\nu_y}(x)dg\nu_y(x).
\end{align}
Here the measures on the fibers $\nu_y$ are defined by the
disintegration $\nu=\int_Y\nu_yd\lambda(y)$.  In our case, the fiber
above $K \in \sub{G}$ is the Poisson boundary $B_K$, and so
Eq.~\ref{eq:rochlin2} becomes Eq.~\ref{eq:fiber-ent}. Since $\lambda$
is $G$-invariant, the entropy of $(\sub{G},\lambda)$ vanishes, and so
Eq.~\ref{eq:rochlin1} becomes Eq.~\ref{eq:bowen-space-entropy}.

A useful property of $h_\mu(B_K,\nu_K)$ is that it is monotone in $K$:
if $K \le H$ then
\begin{align}
  \label{eq:monotonicity}
  h_\mu(B_K,\nu_K) \geq h_\mu(B_H,\nu_H).
\end{align}
This follows from the fact that $(B_H,\nu_H)$ is, in this case, a
factor (as a probability space) of $(B_K,\nu_K)$, and from the
monotonicity of Kullback-Leibler divergence; $h_\mu(B_K,\nu_K)$ is the
$\mu$-expectation of $D_{KL}(\nu_{Kg}||\nu_K)$.

Kaimanovich and Vershik~\cite{kaimanovich1983random} show that the
Furstenberg entropy of the Poisson boundary is equal to the {\em
  random walk entropy}, given by
\begin{align*}
  h_{RW}(G,\mu) = \lim_{n \to \infty}\frac{1}{n}\Ent{Z_n},
\end{align*}
where
\begin{align}
  \label{eq:k-v}
  \Ent{Z_n} =  -\sum_{g \in G}\P{Z_n=g}\log\P{Z_n=g} = \Ent{\mu^n}.
\end{align}
In this spirit, Bowen shows that the entropy of a Bowen space can also
be written as
\begin{align}
  \label{eq:bowen-ent}
  h_\mu(B(\sub{G}), \nu_\lambda) = \lim_{n \to \infty}\frac{1}{n}\int
  \Ent{KZ_n}d\lambda(K) = \inf_n\frac{1}{n}\int \Ent{KZ_n}d\lambda(K),
\end{align}
where
\begin{align*}
  \Ent{KZ_n} = -\sum_{Kg \in K \backslash G}\P{KZ_n=Kg}\log\P{KZ_n=Kg}.
\end{align*}
By the second equality of Eq.~\ref{eq:bowen-ent}, the map $\lambda
\mapsto h_\mu(B(\sub{G}), \nu_\lambda)$ is upper semi-continuous. It
is not, however, continuous in general.

In the next section, in which we discuss lamplighter groups, we give
some examples of Bowen spaces.

\subsubsection{A general bound on the Radon-Nikodym derivatives of the
  Poisson boundaries of induced Markov chains}

The following general lemma, resembling one from Kaimanovich and
Vershik~\cite{kaimanovich1983random}, will be useful below.
\begin{lemma}
  \label{lemma:mc-pb-rn-bound}
  For every $K \in \sub{G}$, $\nu_K$-almost every $x \in B_K$ and
  every $g \in G$ such that $g,g^{-1} \in \supp \mu$ it holds that
  \begin{align*}
    \frac{1}{\mu(g)} \geq \frac{d\nu_{Kg}}{d\nu_K}(x) \geq P_K(Kg, K)
    \geq \mu(g^{-1}).
  \end{align*}
\end{lemma}
\begin{proof}
  Condition on the location of the Markov chain after taking the first
  step, starting at $Kg$. We get for $\nu_K$-almost every $x\in B_K$,
  \begin{align}
    \label{eq:stationary-usual}
    1 = \frac{d\nu_{Kg}}{d\nu_{Kg}}(x)
    = \frac{d\sum_{Kh\in
        K\backslash G}P_K(Kg,Kh)\nu_{Kh}}{d\nu_{Kg}}(x)
    = \sum_{Kh\in
      K\backslash G}P_K(Kg,Kh)\frac{d\nu_{Kh}}{d\nu_{Kg}}(x).
  \end{align}
  Since each summand is positive, it follows that for all $g,h\in G$,
  \begin{align*}
    1\ge P_K(Kg,Kh)\frac{d\nu_{Kh}}{d\nu_{Kg}}(x),
  \end{align*}
  and in particular
  \begin{align*}
    \frac{d\nu_{Kg}}{d\nu_K}(x) \geq P_K(Kg, K).
  \end{align*}
  Note that $P_K(Kg, K) \geq \mu\left(g^{-1}\right)$, by the
  definition of $P_K$.

  Rewriting Eq.~\ref{eq:stationary-usual} as a sum over $G$, we get that
  \begin{align*}
    1 = \frac{d\nu_{Kg}}{d\nu_{Kg}}(x)
    = \frac{d\sum_{g \in G}\mu(g)\nu_{Kg}}{d\nu_{K}}(x)
    = \sum_{g \in G}\mu(g)\frac{d\nu_{Kg}}{d\nu_{K}}(x).
  \end{align*}
  By the same argument above, it follows that
  \begin{align*}
    \frac{1}{\mu(g)} \geq \frac{d\nu_{Kg}}{d\nu_K}(x).
  \end{align*}
  
\end{proof}

\subsection{Lamplighter groups}
\label{sec:lamplighter-intro}
Let $\Gamma$ and $L$ be discrete groups. Let the compact
configurations $\finconf(L, \Gamma)$ be the group of all finitely
supported functions $\Gamma \to L$; that is, if $f \in \finconf(L,
\Gamma)$ then $f$ is equal to the identity of $L$ for all but a finite
number of elements of $\Gamma$. The group operation is pointwise
multiplication:
\begin{align*}
  [f_1f_2](\gamma) = f_1(\gamma)f_2(\gamma),
\end{align*}
and $\Gamma$ acts on $\finconf(L, \Gamma)$ by shifting: 
\begin{align*}
  [\gamma f](\gamma') = f(\gamma^{-1} \gamma').
\end{align*}

The lamplighter group $G=L \wr \Gamma$ is equal to the semidirect
product $\finconf(L, \Gamma) \rtimes \Gamma$, so that the operation is
\begin{align*}
  (f_1, \gamma_1) \cdot (f_2, \gamma_2) = (f_1 (\gamma_1f_2),
  \gamma_1\gamma_2).
\end{align*}
It follows that
\begin{align*}
  (f, \gamma)^{-1} = (\gamma^{-1}f^{-1}, \gamma^{-1}).
\end{align*}
We say that $G$ has {\em base} $\Gamma$ and {\em lamps} in $L$. We
think of the first coordinate as the ``lamp configuration'' and of the
second coordinate as the ``position of the lighter''.

\subsubsection{The limit configuration boundary}
\label{sec:lamp-config}
There exists a natural group homomorphism $\pi : L \wr \Gamma \to
\Gamma$ defined by $\pi(f,\gamma) = \gamma$. We denote $\overline{g} =
\pi(g)$, and $\overline{\mu} = \pi_*\mu$. Thus the $\mu$ random walk
on $L \wr \Gamma$ induces a $\overline{\mu}$ random walk on
$\Gamma$. When $\mu$ has finite first moment, and when the
$\overline{\mu}$ random walk on $\Gamma$ is transient,
Kaimanovich~\cite{kaimanovich1991poisson} shows that the ``value of each
lamp stabilizes'':
\begin{theorem*}[Kaimanovich]
  Let $(Z_1,Z_2,\ldots)$ be a $\mu$ random walk on a finitely
  generated $G=L \wr \Gamma$, let $\mu$ have finite first moment, and
  let the $\overline{\mu}$ random walk on $\Gamma$ be
  transient. Denote $\omega_n=(f_n, \gamma_n)$. Then there exists a
  map $\lampbnd : \Omega \to L^\Gamma$ such that for every $\gamma \in
  \Gamma$
  \begin{align*}
    \lampbnd(\omega_1,\omega_2,\ldots)(\gamma) = \lim_{n \to \infty}f_n(\gamma)
  \end{align*}
  $\mathbb{P}$-almost everywhere.
\end{theorem*}
In particular, each limit $\lim_nf_n(\gamma)$ exists almost surely.

The space of functions $L^\Gamma$ admits the natural left
$\Gamma$-action. As such, it is a $(G,\mu)$-stationary space when
equipped with the measure $\lampbnd_*\mathbb{P}$. Since $\lampbnd$ is
shift invariant, $(L^\Gamma,\lampbnd_*\mathbb{P})$ is a
$(G,\mu)$-boundary, which we shall refer to as the {\em limit
  configuration boundary}. Kaimanovich's theorem equivalently implies
that this boundary has positive entropy, which we denote by
$h_\lamps(G,\mu)$.

\subsubsection{Some Bowen spaces of lamplighters}
In accordance with the definition of $\finconf(L, \Gamma)$ as the set
of finitely supported functions from $\Gamma$ to $L$, let $\finconf(L,
S)$ be the set of finitely supported functions from $\Gamma$ to $L$,
which are supported on $S \subseteq \Gamma$.  Let $K_S$ be the
subgroup of $G$ defined by
\begin{align*}
  K_S = \{(f,e_\Gamma) \in G\,:\, f \in \finconf(L, S)\},
\end{align*}
where $e_\Gamma$ is the identity of $\Gamma$.  The conjugation of
$K_S$ by an element $g=(f,\gamma)$ of $G$ amounts to a shift of $S$ by
$\gamma$, as we show in the next claim.
\begin{claim}
  \label{thm:K-conj}
  Let $(f,\gamma) \in G$. Then $K_S^{(f, \gamma)} = K_{\gamma S}$.
\end{claim}
\begin{proof}
  By definition
  \begin{align*}
    K_S^{(f,\gamma)} = \{(f,\gamma)(g,e_\Gamma)(f,\gamma)^{-1}\,:\, \supp g \subseteq S\}.
  \end{align*}
  Since
  \begin{align*}
    (f,\gamma)(g,e_\Gamma)(f,\gamma)^{-1} = (f (\gamma g),
    \gamma)(\gamma^{-1}f^{-1},\gamma^{-1}) = (f (\gamma g) f^{-1}, e_\Gamma),
  \end{align*}
  it follows that
  \begin{align*}
    K_S^{(f, \gamma)} = \{g\,:\, \supp g \subseteq \gamma S\} =
    K_{\gamma S}.
  \end{align*}
\end{proof}

We call a measure on the subsets of $\Gamma$ a {\em percolation}
measure. The map that assigns the subgroup $K_S < G$ to each $S
\subseteq \Gamma$ maps percolation measures to measures on $\sub{G}$.
It follows from Claim~\ref{thm:K-conj} above that if a percolation
measure $\lambda$ is $\Gamma$-invariant, then the associated measure
on $\sub{G}$ is an IRS measure and therefore $G$-invariant. Likewise, if
$\lambda$ is $\Gamma$-ergodic, then the associated IRS measure, which we also
call $\lambda$, is $G$-ergodic.

Recall that given an ergodic IRS measure $\lambda$, the associated Bowen space
$(B(\sub{G}),\nu_\lambda)$ is also ergodic. We shall use, for our
purposes of entropy realization, Bowen spaces built from ergodic
percolations on $\Gamma$.

As a motivating example, consider the canonical lamplighter $G=(\Z/2\Z)
\wr \Z^d$.  Let $E$ be the set of even elements in $\Z^d$, and let $O$
be its complement, or the set of odd elements. The subgroups $K_E$ and
$K_O$ are, respectively, the finite configurations supported on the
even positions and on the odd positions. By Claim~\ref{thm:K-conj},
conjugation of either of these groups by any element of $G$ either
leaves it invariant or maps it to the other. It follows that $\lambda
= \half\delta_{K_E}+\half\delta_{K_O}$ is an ergodic IRS measure, and that
$(B(\sub{G}), \nu_\lambda)$ is an ergodic Bowen space. This space
consists of two fibers, which are the Poisson boundaries of the
induced Markov chains on $K_E \backslash G$ and on $K_O \backslash G$.

Informally, Eq.~\ref{eq:k-v} states that the entropy of the Poisson
boundary of the random walk on $G$ is equal to the exponential growth
rate of the support of $Z_n$. Intuitively, the growth rate of the
support of $K_EZ_n$, which ``mods out'' the even lamps, should be half
that of the support of $Z_n$, since the random walk entropy of the
projected random walk on $\Gamma$ vanishes. Therefore, by
Eq.~\ref{eq:bowen-ent}, the entropy of $(B(\sub{G}),\nu_\lambda)$ can
be expected to equal half that of the Poisson boundary. By the same
intuition, if we choose an IRS measure in which $K$ includes each lamp
independently with probability $1-p$, then we expect that the entropy
of the associated Bowen space would be $p$ times $h_{RW}(G,\mu)$, the
entropy of the Poisson boundary, and that therefore any entropy in
$[0,h_{RW}(G,\mu)]$ can be realized. We are not able to show this, and
instead resort to a more elaborate construction which only realizes a
dense set of entropies (see Section~\ref{sec:lamplighter}).

For more on invariant random subgroups of lamplighters
see~\cite{bowen2012invariant}.

\subsection{Digression: the Radon-Nikodym compact is not
  necessarily a boundary}
The Radon-Nikodym factor $rn : X \to \R^G$ assigns to almost every
point $x$ in a $(G,\mu)$-stationary space $(X,\nu)$ the function
$f_x(g) = \frac{dg\nu}{d\nu}(x)$. Since this factor commutes with $G$,
its image, called the Radon-Nikodym compact of $(X,\nu)$, is also a
stationary space, which Kaimanovich and Vershik show to have the same
entropy as $(X,\nu)$~\cite{kaimanovich1983random}. An equivalent
definition is given by Nevo and Zimmer~\cite{nevo2000rigidity}.

Consider the example above of the Bowen space
$(B(\sub{G}),\nu_\lambda)$ associated with $\lambda =
\half\delta_{K_E}+\half\delta_{K_O}$. This space is not a boundary,
since it has a factor onto the non-trivial measure preserving space
$(\sub{G},\lambda)$. Furthermore, the map from
$(B(\sub{G}),\nu_\lambda)$ into $(\sub{G},\lambda)$ factors through
the Radon-Nikodym compact, and therefore the compact is also not a
boundary. This is in apparent contradiction to Proposition 3.6
in~\cite{kaimanovich1983random}. Note that counterexamples in Lie
groups appear in~\cite{nevo2000rigidity}, but these do not contradict
the statement of the said proposition, since it is made for discrete
groups only.

\subsection{The support of $\mu$}
Given a generating measure $\mu \in \mathcal{P}(G)$, we construct in
this section a measure $\mup$ supported everywhere on $G$ (and hence
also generating) such that any $G$-space $(X,\nu)$ is
$(G,\mu)$-stationary if and only if it is $(G,\mup)$-stationary, and furthermore
$h_{\mup}(X,\nu) = h_\mu(X,\nu)$. For our purposes of entropy
realization, this will allow us to assume, without loss of generality,
that $\mu$ has full support, which will simplify our proofs.

We first show that if the Poisson boundaries of $(G,\mu)$ and
$(G,\mup)$ coincide then so do their stationary spaces. For this, we
use the characterization of the Poisson boundary via harmonic
functions. Indeed, $\pb(G,\mu)=\pb(G,\mup)$ if and only if every
bounded $\mu$-harmonic function is $\mup$-harmonic, and vice versa. We
next construct a measure $\mup$ that has the same Poisson boundary as
$\mu$.
\begin{lemma}
  \label{lemma:harmonic-stationary}
  Let $\mu, \mup \in \mathcal{P}(G)$ be two generating measures such a
  bounded function $h : G \to \R$ is $\mu$-harmonic if and only if it
  is $\mup$-harmonic. Then a $G$-space $(X,\nu)$ is $\mu$-stationary
  if and only if it is $\mup$-stationary.
\end{lemma}
\begin{proof}
  Let $(X,\nu)$ be $\mu$-stationary, and let $A$ be an arbitrary
  $\nu$-measurable set. Then $h(g) = g\nu(A)$ is $\mu$-harmonic. By
  the claim hypothesis it is also $\mup$-harmonic, and therefore
  \begin{align*}
    \sum_{g \in G}\mup(g)h(g) = h(e).
  \end{align*}
  Hence
  \begin{align*}
    \sum_{g \in G}\mup(g)g\nu(A) = \nu(A),
  \end{align*}
  and since this holds for any $A$ we have that $\mup * \nu = \nu$,
  and $(X,\nu)$ is $\mup$-stationary. The other direction follows by
  symmetry.
\end{proof}

Let $\alpha$ be a measure over the non-negative integers such that
$\alpha(1) \neq 0$. Let
\begin{align*}
  \mup = \sum_{n=0}^\infty\alpha(n)\mu^n,
\end{align*}
where $\mu^n$ denotes the convolution of $\mu$ with itself $n$ times, or
the distribution of $n$ steps of a $\mu$ random walk. Since $\alpha(1)
\neq 0$ then $\mup$ is also generating.

\begin{claim}
  A bounded function $h : G \to \R$ is $\mu$-harmonic if and only if
  it is $\mup$-harmonic.
\end{claim}
\begin{proof}
  It is easy to show that any bounded $\mu$-harmonic function is also
  $\mu^n$-harmonic, and therefore any $\mu$-harmonic function is also
  $\mup$-harmonic, since $\mup$ is a linear combination of convolution
  powers of $\mu$ (see, e.g.,~\cite{kaimanovich1983random}).

  To see the converse, let $h$ be a bounded $\mup$-harmonic function
  on $G$. Let $(Z_1,Z_2,\ldots)$ be a $\mu$ random walk on $G$
  starting from $g$, and denote by $\Esub{g}{\cdot}$ the expectation
  on its probability space. Let $\{\tau_n\}_{n=1}^\infty$ be a random
  walk on $\Z^+$ with transition probabilities $\alpha$. Then
  $(Z_{\tau_1},Z_{\tau_2},\ldots)$ is a coupled $\mup$ random walk on
  $G$, also starting from $g$. Let $M = \lim_nh(Z_{\tau_n})$; note
  that $h(Z_{\tau_n})$ is a bounded martingale w.r.t.\ the filtration
  $\mathcal{F}_n=\sigma(Z_{\tau_1},\ldots,Z_{\tau_n})$, and therefore
  $M$ is well defined.

  To see that $h$ is also $\mu$-harmonic, note that $M$ is measurable
  in the sigma-algebra generated by the union of the following two
  sigma-algebras: $\sigma(Z_1,Z_2,\ldots)$ and the shift-invariant
  sigma-algebra of $\sigma(\tau_1,\tau_2,\ldots)$. However, the latter
  is trivial, as it is the shift-invariant sigma-algebra of an
  aperiodic, irreducible random walk on $\Z^+$. Hence $M$ is
  measurable in $\sigma(Z_1,Z_2,\ldots)$, and so $h'(Z_n) =
  \CondE{M}{Z_n}$ is $\mu$-harmonic. But
  $h'(Z_{\tau_n})=\CondE{M}{Z_{\tau_n}}=h(Z_{\tau_n})$, so $h$ is
  $\mu$-harmonic.
\end{proof}

We have thus, by Lemma~\ref{lemma:harmonic-stationary}, shown that a
$G$-space $(X,\nu)$ is $\mu$-stationary if and only if it is
$\mup$-stationary.  Furthermore, it is easy to
show~\cite{kaimanovich1983random} that
\begin{align*}
  h_{\mu^n}(X,\nu) = n \cdot h_\mu(X,\nu),
\end{align*}
and therefore
\begin{align*}
  h_{\mup}(X,\nu) = h_\mu(X,\nu)\sum_{n=0}^\infty n \alpha(n).
\end{align*}
If we choose $\alpha$ so that $\sum_{n=0}^\infty n \alpha(n) =1$, then
we have that $h_{\mup}(X,\nu) = h_\mu(X,\nu)$.

To summarize, we have shown that $\mu$, $\mup$ share the same
stationary spaces, and that furthermore the set of entropies that can
be realized using ($G$-ergodic) stationary spaces for $\mu$ and $\mup$
are identical. Additionally, it is straightforward to show that if
$\mu$ has finite first moment then so does $\mup$. Therefore, for
the purposes of entropy realization for finite first moment measures,
$\mu$ and $\mup$ are equivalent.

The advantage of $\mup$ is that it is supported everywhere on $G$. We
will henceforth assume, without loss of generality, that $\mu$ is
supported everywhere, which will simplify our proofs.  Note also that
if $\mu$ is supported everywhere then so are its hitting measures on
finite index subgroups, which we discuss in
Section~\ref{sec:subgroups}. Hence all the measures we will concern
ourselves with will be assumed to be supported everywhere.

\section{Entropy realization for lamplighter groups}
\label{sec:lamplighter}
In this section we prove Theorem~\ref{thm:lamplighter-realization}.
To this end, we will prove the following more general proposition, of
which the theorem will be a direct consequence. This proposition will
also be useful to us later.

In Section~\ref{sec:b-ell} we introduce the boundary
$(B_\ell,\nu_\ell)$ of lamplighter groups, which is an extension of
the limit configuration boundary. We denote its entropy by
$h_\ell(G,\mu)$, and so
\begin{align*}
  h_\lamps(G,\mu) \le h_\ell(G,\mu)\le h_{RW}(G,\mu).
\end{align*}
An interesting question is to understand when these numbers are all
equal. In some cases this is known to be true (see
Section~\ref{sec:lamplighter-dense-realization}), and furthermore, the
authors are not aware of any counterexample.

\begin{proposition}
  \label{lem:lamplighter-uniform-realization}
  Let $G=L \wr \Gamma$ be a finitely generated discrete lamplighter
  with base group $\Gamma$ and lamps in $L$. Then there exists a
  family of $G$-ergodic invariant random subgroups
  $\{\lambda_{p,m}\,:\,p \in(0,1),m \in \N\}$, such that, for every
  generating measure $\mu \in \mathcal{P}(G)$ with finite entropy, and
  such that the projected random walk on $\Gamma$ has a trivial
  Poisson boundary, it holds that
  \begin{align*}
    \lim_{m \to \infty} h_\mu(B(\sub{G}),\nu_{\lambda_{p,m}}) =
    p \cdot h_\ell(G,\mu).
  \end{align*}
\end{proposition}

We proceed by deducing Theorem~\ref{thm:lamplighter-realization} from
Proposition~\ref{lem:lamplighter-uniform-realization}, before proving
the proposition itself.  
\begin{proof}[Proof of Theorem~\ref{thm:lamplighter-realization}]
  The statement of
  Proposition~\ref{lem:lamplighter-uniform-realization} is stronger
  than that of the theorem, since $h_\lamps(G,\mu) \leq
  h_\ell(G,\mu)$, and since the family $\left\{\lambda_{p,m}\right\}$
  is universal, in the sense that it can be used to realize entropy
  densely for any finite entropy generating measure on $G$.
\end{proof}

\subsection{Proof of Proposition~\ref{lem:lamplighter-uniform-realization}}

Let $G = L \wr \Gamma$ be a lamplighter group, and let $\mu$ be a
generating measure. Recall that we denote by $\pi$ the projection $G
\to \Gamma$ defined by $\pi(f, \gamma) = \gamma$, and denote
$\overline{\mu} = \pi_*\mu$. We assume that the $\overline{\mu}$
random walk on $\Gamma$ has a trivial Poisson boundary. It follows
that $\Gamma$ is amenable.

Recall (Section~\ref{sec:lamplighter-intro}) the definition of $G$ as
the group $\finconf(L, \Gamma) \rtimes \Gamma$, where $\finconf(L,
\Gamma)$ is the group of finitely supported functions from $\Gamma$ to
$L$, and recall that for $S$ a subset of $\Gamma$, $\finconf(L, S)$ is
the group of finitely supported functions supported on $S$.  Finally
$K_S < G$ is the subgroup of finite configurations supported on $S$,
with the walker in the origin:
\begin{align*}
  K_S = \{(f,e_\Gamma) \in G\,:\, f \in \finconf(L, S)\}.
\end{align*}
A percolation measure $\lambda$ on $\Gamma$ is a measure on subsets of
$\Gamma$. In Section~\ref{sec:lamplighter-intro} above we showed how any 
$G$-invariant ergodic $\lambda$ can be associated, via the map that
assigns the subgroup $K_S$ to the set $S$, with an ergodic Bowen space
$(B(\sub{G}), \nu_\lambda)$.

To prove Proposition~\ref{lem:lamplighter-uniform-realization} we
first, for any $p \in [0,1]$, construct an ergodic percolation measure
$\lambda$ on subsets $S$ of $\Gamma$ such that, with probability close
to $p$, $S$ excludes a large neighborhood of the origin, and with
probability close to $1-p$, $S$ includes a large neighborhood of the
origin.  Hence, the associated IRS measure has the property that
$K_S$, with high probability, either includes or excludes all the
lamps in a large neighborhood of the origin.

Given a percolation measure $\lambda$ on $\Gamma$, we say that
``$\gamma \in \Gamma$ is open'' (or closed) to signify the event that
$\gamma$ is (or is not) an element of the subset drawn from
$\lambda$. Likewise, we say that ``$S \subseteq \Gamma$ is open'' when
all $\gamma \in S$ are open, and that ``$S$ is closed'' when all
$\gamma \in S$ are closed.
\begin{lemma}
  \label{lemma:long-range-percolation}
  Let $\Gamma$ be a discrete amenable group. Then there exists a
  family of percolation measures $\{\lambda_{p,m}\,:\, p \in [0,1],m
  \in \N\}$ that satisfy the following conditions.
  \begin{enumerate}
  \item $\lambda_{p,m}$ is a $\Gamma$-invariant ergodic measure for
    all $p \in [0,1]$ and $m \in \N$.
  \item
    For all $\gamma \in \Gamma$, $p \in [0,1]$ and $m \in \N$ it holds that
    \begin{align*}
      \lambda_{p,m}(\mbox{$\gamma$ is closed}) = p
    \end{align*}
  \item For any finite $S \subset \Gamma$ and all $p \in (0,1)$ it
    holds that
    \begin{align*}
      \lim_{m \to \infty}\lambda_{p,m}(\mbox{$S$ is open}\, \cup \,\mbox{$S$ is closed}) = 1.
    \end{align*}
  \end{enumerate}
\end{lemma}
We prove this lemma in Appendix~\ref{sec:long-range-percolation}
below. For a related result see~\cite{connes1980property}.

The limit $\lim_m\lambda_{p,m}$ is the non-ergodic percolation
$\lambda_p=p\delta_\emptyset+(1-p)\delta_\Gamma$. Clearly, the
Furstenberg entropy of the associated Bowen space is $p\cdot
h_{RW}(G,\mu)$. This is the basic intuition behind
Proposition~\ref{lem:lamplighter-uniform-realization}. However,
$\lambda_p$ is not ergodic, and the map $\lambda \mapsto
h_\mu(B(\sub{G}),\nu_\lambda)$ is not continuous, and therefore the
proof of Proposition~\ref{lem:lamplighter-uniform-realization}
requires some additional work.

Consider two events, namely that $K_S$ either includes or excludes all
the lamps in a large neighborhood of the origin; by
Lemma~\ref{lemma:long-range-percolation} above, the union of these
events nearly covers the probability space.

Consider first the case that $K_S$ includes all the lamps in a large
neighborhood of the origin. Then in the $K_S \backslash G$ Markov
chain, we ``mod out by the lamps of $S$'', so that the states of the
Markov chain do not includes the lamp configuration around the
origin. Therefore, this Markov chain resembles, for the first few
steps, the projected $\overline{\mu}$ random walk on the base group
$\Gamma$, and therefore the entropy $h_\mu(B_{K},\nu_{K})$ could be
expected to be low. Conversely, when $K_S$ excludes all the lamps in a
large neighborhood of the origin, the $K_S \backslash G$ chain
includes all the information about the lamps around the origin, and
therefore, in the first few steps, resembles the $\mu$ random walk on
$G$, and thus $h_\mu(B_{K},\nu_{K})$ could be expected to have entropy
that is close to that of the Poisson boundary, or at least that of the
limit configuration boundary.  This intuition is formalized in the
following two lemmas, which we prove below.
\begin{lemma}
  \label{lem:close-to-full}
  Let $\{S_r\}_{r=1}^\infty$ be a sequence of cofinite subsets of
  $\Gamma$ such that $\lim_rS_r = \emptyset$, in the topology of
  convergence on finite sets. Then there exists an $h_\ell(G,\mu) > 0$
  such that
  \begin{align*}
    \frac{1}{n} \lim_{r\to\infty}h_{\mu^n}\left(B_{K_{S_r}},\nu_{K_{S_r}}\right)=
    h_\ell(G,\mu),
  \end{align*}
  for all $n \ge 1$.
\end{lemma}

\begin{lemma}
  \label{lem:close-to-base}
  Let $\{S_r\}_{r=1}^\infty$ be a sequence of finite subsets of
  $\Gamma$ such that $\lim_rS_r = \Gamma$, in the topology of
  convergence on finite sets. Then
  \begin{align*}
    \lim_{n\to\infty}\frac{1}{n}\lim_{r\to\infty}h_{\mu^{n}}\left(B_{K_{S_r}},\nu_{K_{S_r}}\right)=0.
  \end{align*}
\end{lemma}

The following corollary is a direct consequence of these two lemmas.
\begin{corollary}
  \label{cor:2-lemmas}

  For every $\eps>0$ there exists a finite set $S\subset \Gamma$ and
  $n \in \N$ such that both
  \begin{align*}
    \frac{1}{n} h_{\mu^n}\left(B_{K_{S}},\nu_{K_{S}}\right) \leq \eps
     \quad\quad\mbox{and}\quad\quad
     \frac{1}{n} h_{\mu^n}\left(B_{K_{S^c}},\nu_{K_{S^c}}\right) \geq h_\ell(G,\mu) - \eps,
  \end{align*}
  where $S^c$ is the complement of $S$ in $\Gamma$.

\end{corollary}

We are now ready to prove
Proposition~\ref{lem:lamplighter-uniform-realization}. The idea of the
proof is as follows.  By Corollary~\ref{cor:2-lemmas}, for finite $S$
large enough, the entropy of $(B_{K_S},\nu_{K_S})$, the Poisson
boundary of the induced Markov chain on $K_S \backslash G$, is close
to zero whenever the event ``$S$ is open'' occurs.  On the other hand,
if the event ``$S$ is closed'' occurs, then the entropy is close to
$h_\ell(G,\mu)$. Now, using Lemma~\ref{lemma:long-range-percolation},
we can find a percolation measure such that the event ``$S$ is open''
occurs with probability almost $1-p$ and ``$S$ is closed'' occurs with
probability almost $p$. It follows that the entropy is close to $p
\cdot h_\ell(G,\mu)$.

\begin{proof}[Proof of Proposition~\ref{lem:lamplighter-uniform-realization}]
  Fix $p\in(0,1)$. Let $\{\lambda_{p,m}\}$ be the set of measures
  defined in Lemma~\ref{lemma:long-range-percolation}, and let
  $(B(\sub{G}),\nu_{\lambda_{p,m}})$ be the Bowen space associated
  with $\lambda_{p,m}$; we here identify the percolation measure
  $\lambda_{p,m}$ with the associated IRS measure, and denote the
  latter too by $\lambda_{p,m}$.

  We shall prove the claim by showing that for every $\eps>0$, and for
  every $\mu\in\mathcal{P}(G)$ that satisfies the conditions of the
  claim, it holds that for $m$ large enough
  \begin{align*}
    p \cdot h_\ell(G,\mu)-\eps \leq h_\mu(B(Sub_G),\nu_{\lambda_{p,m}}) \leq
    p \cdot h_\ell(G,\mu)+\eps.
  \end{align*}  
  
  Let $\eps>0$ and let $\mu \in \mathcal{P}(G)$ satisfy the conditions
  of the claim. Denote $h_\ell = h_\ell(G,\mu)$, and let $S$ be a
  finite subset of $\Gamma$ such that
  \begin{align*}
    \frac{1}{n} h_{\mu^n}\left(B_{K_{S}},\nu_{K_{S}}\right) \leq
    \eps
    \quad\quad \mbox{and}\quad\quad
    \frac{1}{n} h_{\mu^n}\left(B_{K_{S^c}},\nu_{K_{S^c}}\right)
    \geq h_\ell(G,\mu) - \eps,
  \end{align*}
  where $S^c$ is the complement of $S$ in $\Gamma$.  The existence of
  this set is guaranteed by Corollary~\ref{cor:2-lemmas}. Let $\eps'$
  be determined later, and apply
  Lemma~\ref{lemma:long-range-percolation} to $S$ to find $m\in \N$
  large enough such that
  \begin{align*}
     (1-p)-\eps' \leq \lambda_{p,m} (\mbox{$S$ is open}) \leq 1-p
  \end{align*}
  and
  \begin{align*}
    p-\eps' \leq \lambda_{p,m} (\mbox{$S$ is closed}) \leq p.
  \end{align*}

  Recall that the $\mu^n$ Furstenberg entropy of
  $(B(\sub{G}),\nu_{\lambda_{p,m}})$ is given by
  \begin{align*}
    h_{\mu^n}\left(B(\sub{G}),\nu_{\lambda_{p,m}}\right) =
    \int_{\sub{G}}h_{\mu^n}(B_K,\nu_K) d\lambda_{p,m}(K).
  \end{align*}
  We shall integrate separately over the event that $S$ is open
  ($\cA_1$), the event that $S$ is closed ($\cA_2$) and the
  complements of these events ($\cA_3$) and bound the integral over
  $\cA_1$ and $\cA_3$ from above and over $\cA_2$ from both above and
  below.

  By the definitions of these events, for every $K_{S_1} \in \cA_1$ it
  holds that $S$ is a subset of $S_1$, and for every $K_{S_2} \in
  \cA_2$ it holds that the complement of $S$ is a superset of
  $S_2$. Hence, by Corollary~\ref{cor:2-lemmas}, and by the
  monotonicity of $h_\mu(B_K,\nu_K)$ (Eq.~\ref{eq:monotonicity}), we
  have that
  \begin{align}
    \label{eq:A1}
    \frac{1}{n} \int_{\cA_1}h_{\mu^n}(B_K,\nu_K)d\lambda_{p,m}(K)
    &\le \frac{1}{n} \int_{\cA_1}h_{\mu^n}(B_{K_S},\nu_{K_S})d\lambda_{p,m}(K) \notag \\
    &= \frac{1}{n}
    h_{\mu^n}(B_{K_S},\nu_{K_S})\int_{\cA_1}d\lambda_{p,m}(K) \notag \\
    &\le \eps \cdot \lambda_{p,m}(\cA_1)\notag\\
    &\le \eps \cdot (1-p).
  \end{align}
  and
  \begin{align}
    \label{eq:A2-b}
    \frac{1}{n} \int_{\cA_2}h_{\mu^n}(B_K,\nu_K) d\lambda_{p,m}(K)
    &\geq
    \frac{1}{n} \int_{\cA_2}h_{\mu^n}(B_{K_{S^c}},\nu_{K_{S^c}}) d\lambda_{p,m}(K) \notag\\
    &=
    \frac{1}{n}h_{\mu^n}(B_{K_{S^c}},\nu_{K_{S^c}}) \int_{\cA_2}
    d\lambda_{p,m}(K) \notag \\
    &\geq (h_{\ell}-\eps)
    \cdot \lambda_{p,m}(\cA_2)  \notag\\
    &\ge (h_{\ell}-\eps) \cdot (p-\eps').
  \end{align}

  By Claim~\ref{claim:sigma-algebras}, and using the monotonicity of
  Kullback-Leibler divergence, $h_{\mu^n}(B_K,\nu_K)$ is uniformly
  bounded by $n\cdot h_\ell$.  Hence
  \begin{align}
    \label{eq:A3}
    \frac{1}{n}\int_{\cA_3}h_{\mu^n}(B_K,\nu_K) d\lambda_{p,m}(K) \le
    h_\ell \cdot \lambda_{p,m}(\cA_3) \leq h_{\ell} \cdot 2\eps',
  \end{align}
  as $\lambda_{p,m}(\cA_3) \leq 2\eps'$. Likewise
  \begin{align}
    \label{eq:A2-a}
    \frac{1}{n}\int_{\cA_2}h_{\mu^n}(B_K,\nu_K)
    d\lambda_{p,m}(K) \le h_{\ell} \cdot \lambda_{p,m}(\cA_2) \le h_{\ell} \cdot p.
  \end{align}
    
  Collecting terms and substituting $\eps' = \eps\cdot p \cdot
  (1-p)/(2h_\ell)$, we get that
  \begin{align*}
    \frac{1}{n}\int_{Sub_G}h_{\mu^n}(B_K,\nu_K)
    d\lambda_{p,m}(K) \in \left[ p \cdot h_\ell -\eps,
    p \cdot h_\ell + \eps\right].
  \end{align*}  
  But the left hand side is equal to $h_\mu(B(\sub{G}),\lambda_{p,m})$, and so
  the proof is complete.
\end{proof}

\subsection{The boundary $(B_\ell,\nu_\ell)$ and a proof of
  Lemma~\ref{lem:close-to-full}}
\label{sec:b-ell}

To prove Lemma~\ref{lem:close-to-full}, we introduce the boundary
$(B_\ell,\nu_\ell)$ of the $(G,\mu)$ random walk. We show that this
boundary is an extension of the limit configuration boundary, and as
such has entropy $h_\ell(G,\mu) = h_\mu(B_\ell,\nu_\ell) \geq
h_\lamps(G,\mu)$. We then show that when $S$ excludes a large neighborhood of
lamps around the origin then $h_\mu(B_{K_S},\nu_{K_S})$ is close to
$h_\ell(G,\mu)$.  While it may be the case that $(B_\ell,\nu_\ell)$ is, in
fact, equal to $\pb(G,\mu)$, the Poisson boundary of $(G,\mu)$, we are
not able to show this in full generality (see the discussion below in
Section~\ref{sec:lamplighter-dense-realization}).

Let $S$ be a cofinite subset of $\Gamma$.  A state of the $\mu$ Markov
chain $K_S \backslash G$ can be thought of as the pair consisting of
the finite configuration of the lamps outside $S$, and the position of
the lighter.

Let $(B_{K_S},\nu_{K_S})$ be the Poisson boundary of the $\mu$ Markov
chain on $K_S \backslash G$. Note that $B_{K_S}$ is not a $G$-space,
but it is a factor of $\pb(G,\mu)$, and as such we can think of it as
a sub-sigma-algebra $\mathcal{F}_S$ of the shift invariant
sigma-algebra on $\Omega$, which, however, is not $G$-invariant. We
define $\mathcal{F}_\ell$ as the sigma-algebra generated by the union
of all these sigma-algebras:
\begin{align*}
  \mathcal{F}_\ell = \sigma\left(\bigcup_{|S^C| < \infty}\mathcal{F}_S\right).
\end{align*}
Note that while $\mathcal{F}_S$ is not $G$-invariant,
$\mathcal{F}_\ell$ is, since $g=(f,\gamma)$ acts on $\mathcal{F}_S$ by
shifting it to $\mathcal{F}_{\gamma S}$. Therefore $\mathcal{F}_\ell$
is $G$-invariant and shift invariant, and therefore its Mackey
realization, which we denote by $(B_\ell,\nu_\ell)$, is a
$(G,\mu)$-boundary. As such, it is a factor of $\pb(G,\mu)$. Denote
its entropy by $h_\ell(G,\mu) = h_\mu(B_\ell,\nu_\ell)$.

Since the limit configuration boundary is generated by cylinders of the
final states of finite sets of lamps, then it is a factor of
$B_\ell$. The following claim is a consequence of these definitions.
\begin{claim}
  \label{thm:h-ell-h-lamps}
  \begin{align*}
    h_\ell(G,\mu) \geq h_\lamps(G,\mu).
  \end{align*}
\end{claim}

Let $S_r\subset\Gamma$ be a sequence of cofinite subsets such that
$\lim_rS_r=\emptyset$. Consider the sequence of subgroups
$K_r=K_{S_r}$, and let $\mathcal{F}_r$ be the sigma-algebra of the
Poisson boundary of the $\mu$ Markov chain on $K_r \backslash G$.
\begin{claim}
  \label{claim:sigma-algebras}
  \begin{align*}
    \mathcal{F}_\ell = \sigma\left(\bigcup_{r=1}^\infty\mathcal{F}_r\right).
  \end{align*}
\end{claim}
\begin{proof}
  Denote
  $\mathcal{F}_\infty=\sigma\left(\cup_{r=1}^\infty\mathcal{F}_r\right)$,
  and recall that $\mathcal{F}_\ell = \sigma\left(\cup_{|S^C| < \infty}\mathcal{F}_S\right)$.  Since
  $\mathcal{F}_r=\mathcal{F}_{S_r}$, where $S_r$ is finite, it follows
  that $\mathcal{F}_r \subseteq \mathcal{F}_\ell$, and so
  $\cup_{r=1}^\infty\mathcal{F}_r \subseteq \mathcal{F}_\ell$ and
  \begin{align*}
    \mathcal{F}_\infty =
    \sigma\left(\cup_{r=1}^\infty\mathcal{F}_{S_r}\right) \subseteq
    \mathcal{F}_\ell.
  \end{align*}

  Conversely, note that for each finite $S \subset \Gamma$ there
  exists an $r$ such that $S$ is a subset of the complement in
  $\Gamma$ of $S_r$, since $\lim_rS_r \to \emptyset$. Hence
  $\mathcal{F}_S \subseteq \mathcal{F}_{S_r} \subseteq
  \mathcal{F}_\infty$, $\cup_{|S^C| < \infty}\mathcal{F}_S \subseteq
  \mathcal{F}_\infty$, and the claim follows.
\end{proof}

Let $\bnd_\ell : \Omega \to B_\ell$ be the boundary map associated
with $B_\ell$, let $\Omega_r=(K_r\backslash G)^\N$, let
$\bnd_r:\Omega_r\to B_{K_r}$ be the boundary map of the induced Markov
chain on $K_r \backslash G$, and let $\pi_r:B_\ell \to B_{K_r}$ be the
natural factor. Then $\pi_r$ is similar to $G$-equivariant maps, in
the sense that $\pi_{r*}g\nu_\ell = \nu_{K_rg}$ for all $g \in G$.
\begin{claim}
  \label{thm:rn-converges}
  For a fixed $g \in G$, and $\nu_\ell$-almost every $b \in B_\ell$,
  \begin{align*}
    \lim_{r\to\infty}\frac{d\nu_{K_r}}{d\nu_{K_rg}}(\pi_r(b))=\frac{d\nu_\ell}{dg\nu_\ell}(b).
  \end{align*}
\end{claim}
\begin{proof}
  Since
  \begin{align*}
    \CondP{Z_1=g}{\bnd_r(Z_1,Z_2,\ldots)=\pi_r(b)} = \P{Z_1=g}\frac{d\nu_{K_rg}}{d\nu_{K_r}}(b)
  \end{align*}
  and
  \begin{align*}
    \CondP{Z_1=g}{\bnd_\ell(Z_1,Z_2,\ldots)=b} = \P{Z_1=g}\frac{dg\nu_\ell}{d\nu_\ell}(b),
  \end{align*}
  it is enough to show that for $\nu_\ell$-almost every $b\in B_\ell$,
  \begin{align*}
    \lim_{r\to\infty}\CondP{Z_1=g}{\bnd_r(Z_1,Z_2,\ldots)=\pi_r(b)} =
    \CondP{Z_1=g}{\bnd_\ell(Z_1,Z_2,\ldots)=b}.
  \end{align*}
  This, however, is a consequence of Claim~\ref{claim:sigma-algebras},
  and the claim follows.
\end{proof}
\begin{claim}
  \label{clm:lim_kl}
  For a fixed $g \in G$,
  \begin{align*}
    \lim_{r\to\infty}D_{KL}\left(\nu_{K_rg}||\nu_{K_r}\right)=D_{KL}\left(g\nu_\ell||\nu_\ell\right).
  \end{align*}
\end{claim}
\begin{proof}
  Since $\nu_{K_rg}=\pi_{r*}(g\nu_\ell)$ we get that 
  \begin{align*}
    D_{KL}\left(\nu_{K_rg}||\nu_{K_r}\right) &=
    \int_{B_{K_r}}-\log\frac{d\nu_{K_r}}{d\nu_{K_rg}}(x)d\nu_{K_rg}(x)\\
    & = \int_{B_\ell}-\log\frac{d\nu_{K_r}}{d\nu_{K_rg}}(\pi_r(b))dg\nu_\ell(b).
  \end{align*}

  By Claim~\ref{thm:rn-converges}, the functions $f_r(b) =
  -\log\frac{d\nu_{K_r}}{d\nu_{K_rg}}(\pi_r(b))$ converge pointwise to
  $f(b)=-\log\frac{d\nu_\ell}{dg\nu_\ell}(b)$. Hence the claim will
  follow by the dominated convergence theorem, provided that we can
  show that the functions $f_r$ are uniformly bounded. By
  Lemma~\ref{lemma:mc-pb-rn-bound} we have that
  \begin{align}
    \label{eq:rn-bound}
    \mu\left(g^{-1}\right)\le\frac{d\nu_{K_rg}}{d\nu_{K_r}}(x)\le\frac{1}{\mu(g)}.
  \end{align}

  Since $\mu$ is without loss of generality supported everywhere (see
  Section~\ref{sec:stationary}),  $\mu(g) > 0$ and
  $\mu\left(g^{-1}\right)>0$, and so
  \begin{align*}
    -\log\mu\left(g^{-1}\right) \ge
    -\log\frac{d\nu_{K_rg}}{d\nu_{K_r}}(x) \ge -\log\frac{1}{\mu\left(g\right)}.
  \end{align*}
  Therefore the functions $f_r$ are uniformly bounded.
\end{proof}

\begin{proof}[Proof of Lemma~\ref{lem:close-to-full}]
  We prove the lemma for $n=1$, and note that for arbitrary $n$ the
  proof follows by the same argument, since
  $h_{\mu^n}(B_\ell,\nu_\ell) = n \cdot h_\ell(G,\mu)$.
  
  Note that by the monotonicity of Kullback-Leibler divergence,
  $D_{KL}\left(\nu_{K_rg}||\nu_{K_r}\right)\le
  D_{KL}\left(g\nu_\ell||\nu_\ell\right)$.  Note also that the
  finiteness of $h_\ell(G,\mu)$ implies that
  $D_{KL}\left(g\nu_\ell||\nu_\ell\right)$ is $\mu$-integrable, as
  function of $g$. Hence, by the dominated convergence theorem,
  \begin{align*}
    \lim_{r\to\infty}h_{\mu}(B_{K_r},\nu_{K_r})
    &= \lim_{r\to\infty}\sum_{g\in G}\mu(g)D_{KL}(\nu_{K_rg}||\nu_{K_r})\\
    & = \sum_{g\in G}\mu(g)\lim_{r\to\infty}D_{KL}(\nu_{K_rg}||\nu_{K_r})\\
    & = \sum_{g\in G}\mu(g)D_{KL}(g\nu_\ell||\nu_\ell)\\
    & = h_\ell(G,\mu),
  \end{align*}
  where the third equality follows from Claim~\ref{clm:lim_kl}.
\end{proof}

\subsection{Proof of Lemma~\ref{lem:close-to-base}}

Let $\{S_r\}_{r=1}^\infty$ be a sequence of subsets of $\Gamma$ with
$\lim_rS_r=\Gamma$, consider the sequence of subgroups $K_r=K_{S_r}$,
and consider the induced Markov chains on $K_{r}\backslash G$. As
$r\to\infty$, we are ``modding out by more and more lamps'', and so
the Markov chains resemble more and more closely the projected walk on
$\Gamma$ itself, which has zero entropy, since it has a trivial
Poisson boundary. Indeed, in this section we prove that the entropies
$h_\mu(B_{K_r},\nu_{K_r})$ converge to zero.

Recall (Section~\ref{sec:bowen-spaces}) that $P_K(Kg,Kh)$ is the
transition probability from $Kg$ to $Kh$ in the induced Markov chain
on $K \backslash G$.  Recall also that the projection $\pi : G \to \Gamma$
is defined by $\pi(f,\gamma) = \gamma$, and that we denote
$\overline{g} = \pi(g)$ and $\overline{\mu} = \pi_*\mu$.
\begin{claim}
  \label{thm:transition-probs}
  For all $g\in G$ it holds that
  \begin{align*}
    \lim_{r\to\infty}P_{K_r}(K_{r},K_{r}g) = \overline{\mu}\left(\overline{g}\right).
  \end{align*}
\end{claim}
It follows directly that
\begin{align*}
  \lim_{r\to\infty}P_{K_r}^n(K_{r},K_{r}g) = \overline{\mu}^n\left(\overline{g}\right).
  \end{align*}
\begin{proof}
  Recall that $\finconf(L, S)$ is the set of finite lamp
  configurations supported on $S$. By definition,
  \begin{align*}
    P_{K_r}(K_r,K_rg) &=\sum_{k\in K_r}\mu\left(kg\right)\\
    &=  \sum_{f\in \finconf(L, S_r)}\mu\left((f, e_\Gamma)g\right)
  \end{align*}

  Observe that $K_\Gamma$ is a normal subgroup in $G$ and that
  $\Gamma=K_\Gamma \backslash G$. Hence
  \begin{align*}
    \overline{\mu}\left(\overline{g}\right) &=  \sum_{f\in
      \finconf(L, \Gamma)}\mu\left((f,e_\Gamma)g\right)\\
    &= \sum_{f\in \finconf(L, \Gamma)}\mu\left(g(f,e_\Gamma)\right).
  \end{align*}

  Now, since $\lim_rS_r=\Gamma$, it follows that $\lim_r\finconf(L,
  S_r)= \finconf(L, \Gamma)$ and hence
  \begin{align*}
    \lim_{r\to\infty}P_{K_{r}}\left(K_{r},K_{r}g\right)
    & =  \lim_{r\to\infty}\sum_{f\in \finconf(L,S_{r})}\mu\left(\left(f,e_\Gamma\right)g\right)\\
    & =  \sum_{f\in \finconf(L,\Gamma)}\mu\left(\left(f,e_\Gamma\right)g\right)\\
    & =  \overline{\mu}\left(\overline{g}\right).
  \end{align*}
\end{proof}

\begin{proof}[Proof of Lemma~\ref{lem:close-to-base}]
  For fixed $n,r\in\mathbb{N}$,
  \begin{align*}
    h_{\mu^n}\left(B_{K_{r}},\nu_{K_{r}}\right)
    &= \sum_{g\in G}\mu^n(g)\int_{B_{K_r}}-\log\frac{d\nu_{K_r}}{d\nu_{K_rg}}(b)d\nu_{K_rg}(b)\\
    & \le \sum_{g\in G}\mu^n(g)\int_{B_{K_r}}-\log P_{K_r}^n(K_r,K_rg)d\nu_{K_rg}(b)\\
    & = \sum_{g\in G}\mu^n(g)\cdot-\log P_{K_r}^n(K_r,K_rg),
  \end{align*}
  where the inequality is an application of
  Lemma~\ref{lemma:mc-pb-rn-bound}.

  By Claim~\ref{thm:transition-probs}, $\lim_rP_{K_r}^n(K_r,K_rg) =
  \overline{\mu}(\overline{g})$. By $0\le-\log
  P^n_{K_r}\left(K_{r},K_{r}g\right) \le -\log\mu^{n}\left(g\right)$ and
  the finiteness of $H\left(\mu\right)$ we can use the dominated
  convergence theorem to arrive at
  \begin{align*}
    \frac{1}{n}\lim_{r\to\infty}h_{\mu^n}\left(B_{K_r},\nu_{K_r}\right)
    & \le \frac{1}{n}\sum_{g\in G}\mu^{n}(g)\cdot-\log\overline{\mu}^{n}\left(\overline{g}\right)\\
    & = \frac{1}{n}\sum_{\gamma\in\Gamma}-\overline{\mu}^{n}(\gamma)\log\overline{\mu}^{n}(\gamma)
  \end{align*}
  Taking the limit as $n$ tends to infinity we get that
  \begin{align*}
    \lim_{n \to
      \infty}\frac{1}{n}\lim_{r\to\infty}h_{\mu^{n}}\left(B_{K_{r}},\nu_{K_{r}}\right)
    \le  \lim_{n \to \infty}\frac{1}{n} \Ent{\overline{\mu}^n} = 0,
  \end{align*}
  where the final equality is again a consequence of the fact that
  the $\overline{\mu}$ random walk on $\Gamma$ has a trivial Poisson
  boundary~\cite{kaimanovich1983random}.
\end{proof}

\subsection{Dense entropy realization for some lamplighter groups}
\label{sec:lamplighter-dense-realization}

The boundary $B_\ell$ is an extension of the limit configuration
boundary and a factor of the Poisson boundary. It follows that when
the limit configuration boundary is equal to the Poisson boundary than
so is $B_\ell$, and $h_\ell = h_{RW}$.

Thus, a direct consequence of
Theorem~\ref{thm:lamplighter-realization} is the following theorem.
\begin{theorem}
  \label{thm:lamplighter-dense-realization}
  Let $G = L \wr \Gamma$ with non-trivial $L$, let $\mu$ generate $G$,
  and let the $(G,\mu)$ limit configuration boundary equal its Poisson
  boundary. Then there exists a dense set $H \subseteq
  [0,h_{RW}(G,\mu)]$ such that for every $h \in H$ there exists a
  $(G,\mu)$ ergodic Bowen space $(X,\nu)$ with $h_\mu(X,\nu) = h$.
\end{theorem}

The relation between the Poisson boundary and the limit configuration
boundary is a subject of active
research. Kaimanovich~\cite{kaimanovich2000poisson} shows that the
Poisson boundary coincides with the limit configuration boundary on
$((\Z/2\Z) \wr Z^d, \mu)$ when $\mu$ has a first moment and the
projected random walk on $\Z^d$ has a
drift. Erschler~\cite{erschler2011poisson} shows that these boundaries
are equal for $((\Z/2\Z) \wr \Z^d, \mu)$, when $d \geq 5$ and $\mu$
has with finite third moment. An additional equality result on
non-amenable base groups is given by Karlsson and
Woess~\cite{karlsson2007poisson}.

\section{No entropy gap for virtually free groups}
\label{sec:lifting}
In this section we prove our main result,
Theorem~\ref{thm:main-realization}, which states that when $G$ is
virtually free and $\mu$ has finite first moment then $(G,\mu)$ does
not have an entropy gap.

The general idea is to ``lift'' the no entropy gap result from
lamplighters to virtually free groups. First, we establish this result
for free groups in Section~\ref{sec:free-groups}, and then lift it to
finite index supergroups.

To relate the stationary actions of a group and a finite index
subgroup, we consider the hitting measure on the subgroup.  In
Section~\ref{sec:hitting-measures}, we show that measures with finite
first moment have hitting measures with finite first moment. 

Then, in Section~\ref{sub:Lifting-Bowen-spaces} we discuss a standard
construction which lifts IRS measures from finite index subgroups, and
hence also lifts the associated Bowen spaces.  For the case that the
finite index subgroup is normal, we relate, in
Section~\ref{sec:entropy-of-lifting}, the entropies these Bowen
spaces.  Finally, in Section~\ref{sec:proof-of-main-thm}, we bring
these ideas together to prove a no entropy gap result for virtually
free groups.

\subsection{From lamplighters to free groups}
\label{sec:free-groups}
In this section we show that as extensions of lamplighters, free
groups admit a result that is parallel to
Proposition~\ref{lem:lamplighter-uniform-realization}, the stronger version
of Theorem~\ref{thm:lamplighter-realization} that we proved above.

The following claim is standard.
\begin{claim}
  \label{lem:quotient-preserves-ffm}
  Let $G$ be finitely generated, and let $G\xrightarrow{\varphi}Q$ be
  a group homomorphism onto $Q$. If $\mu\in\mathcal{P}(G)$ has finite
  first moment then $\varphi_{*}\mu\in\mathcal{P}\left(Q\right)$ has
  finite first moment.
\end{claim}

\begin{lemma}
  \label{lem:quotient-preserves-entropy}
  Let $G$ be finitely generated, let $G\xrightarrow{\varphi}Q$ be a
  group homomorphism onto $Q$, and let $\mu\in\mathcal{P}(G)$. Let
  $(B(\sub{Q}),\nu_\lambda)$ be an ergodic $(Q,\varphi_*\mu)$ Bowen
  space. Then $(B(\sub{G}),\nu_{\varphi^{-1}\lambda})$ is an ergodic
  $(G,\mu)$ Bowen space, and furthermore
  \begin{align*}
    h_\mu(B(\sub{G}),\nu_{\varphi^{-1}\lambda}) = h_{\varphi_{*}\mu}(B(\sub{Q}),\nu_\lambda).    
  \end{align*}
\end{lemma}
\begin{proof}
  Note that $G$ acts naturally on $\sub{Q}$ through $\varphi$. Hence
  $\sub{Q}$ is a $G$-space.  Since $\varphi^{-1} : \sub{Q} \to
  \sub{G}$ is $G$-equivariant, then $(\sub{G}, \varphi^{-1}_*\lambda)$
  is a $G$-factor of $(\sub{Q},\lambda)$. Since the latter is
  invariant and ergodic, it follows that the former is too, and hence
  is an ergodic $G$ IRS measure. Furthermore, since $\varphi(\varphi^{-1}(K))
  = K$, the two spaces are $G$-isomorphic.

  The same can also be said for the spaces of induced random walks on
  $K \backslash Q$ and $\varphi^{-1}(K) \backslash G$, and therefore
  the induced Markov chains are also isomorphic. Finally, since, as a
  map between $(K \backslash Q)^\N \to (\varphi^{-1}(K) \backslash
  G)^\N$, $\varphi^{-1}$ is shift invariant, it follows that $(B_K,
  \nu_K)$ is $G$-isomorphic to
  $(B_{\varphi^{-1}(K)},\nu_{\varphi^{-1}(K)})$, and so the Bowen
  spaces $(B(\sub{Q}),\nu_{\varphi_*\lambda})$ and
  $(B(\sub{G}),\nu_{\varphi^{-1}\lambda})$ are $G$-isomorphic.
  
  To see the equality in entropies, note that in general, every
  $(Q,\varphi_*\mu)$-stationary space is also $(G,\mu)$-stationary,
  and
  \begin{align*}
    h_{\mu}(X,\nu) = h_{\varphi_*\mu}(X,\nu).
  \end{align*}
\end{proof}

Consider the canonical lamplighter $(\Z/2\Z) \wr \Z^3$. Since any
generating random walk on $\Z^3$ is transient (see., e.g., Proposition
3.20 in~\cite{woess2000random}), by
Kaimanovich~\cite{kaimanovich1991poisson}, for any finite first moment
$\mu$, it holds that $h_\ell((\Z/2\Z) \wr \Z^3,\mu) \ge
h_\lamps((\Z/2\Z) \wr \Z^3,\mu) >0$. Therefore,
Claim~\ref{lem:quotient-preserves-ffm} and
Lemma~\ref{lem:quotient-preserves-entropy}, together with
Proposition~\ref{lem:lamplighter-uniform-realization}, yield the
following proposition.
\begin{proposition}
  \label{lem:freegroup-uniform-realization}
  Let $G$ be a finitely generated extension of $(\Z/2\Z) \wr \Z^3$,
  with $\varphi : G \to (\Z/2\Z) \wr \Z^3$ the quotient map. Then
  there exists a family of $G$-ergodic invariant random subgroup
  measures $\{\lambda_{p,m}\,:\,p \in(0,1),m \in \N\}$ such that, for
  every generating measure $\mu \in \mathcal{P}(G)$ with finite first
  moment it holds that
  \begin{align*}
    \lim_{m \to \infty} h_\mu(B(\sub{G}),\nu_{\lambda_{p,m}}) =
    p \cdot h_\ell(\varphi G, \varphi_*\mu),
  \end{align*}
  where $h_\ell(\varphi G, \varphi_*\mu) > 0$.  
\end{proposition}
Since $(\Z/2\Z) \wr \Z^3$ can be generated as a group by a set of
four generators, this holds for $F_n$, with $n \geq 4$.

\subsection{Hitting measures and finite first moments}
\label{sec:hitting-measures}
\begin{lemma}
  \label{thm:theta-is-finite-first-moment}

  Let $G$ be a finitely generated group, and let $\Gamma\le G$ with
  $[G:\Gamma]<\infty$. If $\mu\in\mathcal{P}(G)$ has finite first
  moment, then the hitting measure $\theta\in \mathcal{P}(\Gamma)$
  also has finite first moment.
\end{lemma}

We prove this lemma in Appendix~\ref{sec:hitting-measures-proof}.

By Proposition~\ref{lem:freegroup-uniform-realization} and
Lemma~\ref{thm:theta-is-finite-first-moment} we conclude the
following.  Let $G$ be a group with $F_n$ as a finite index
subgroup. Let $\mu\in\mathcal{P}(G)$ be a finite first moment
generating measure, and consider its hitting measure
$\theta\in\mathcal{P}(F_n)$. Then $(F_n,\theta)$ has no entropy gap.
In the next sections, we use this fact to prove our no entropy gap
result for $(G,\mu)$.

\subsection{Lifting Bowen spaces from lattices}
\label{sub:Lifting-Bowen-spaces}

The following construction applies to a more general setting, where
$G$ is a locally compact group and $\Gamma$ is a lattice in $G$ (see,
e.g.,~\cite{stuck1994stabilizers}).  That is, there exists a
$G$-invariant measure $\eta \in \mathcal{P}(G/\Gamma)$.

Denote by $\irs\left(\Gamma\right)$ the set of all $\Gamma$ invariant
random subgroup measures.  Let $\lambda\in \irs\left(\Gamma\right)$. Then
$\lambda$ is $\Gamma$-invariant but not, in general,
$G$-invariant. Note, however, that if $g_1\Gamma=g_2\Gamma$ then there
exists a $\gamma\in\Gamma$ such that $g_1=g_2\gamma$. Hence
$g_1\lambda=g_2\gamma\lambda=g_2\lambda$. Therefore, the $G$-action on
$\lambda$ is constant on cosets of $\Gamma$, and the measure
$(g\Gamma)\lambda$ is well defined for every $g\Gamma \in G / \Gamma$.

Denote by $\eta*\lambda$ the measure
\begin{align*}
  \eta*\lambda = \int_{G/\Gamma}(g\Gamma)\lambda
  d\eta(g\Gamma).
\end{align*}
The following claim is straightforward.
\begin{claim}
  \label{clm:lifting-irs}
  If $\lambda \in \irs(\Gamma)$ then $\eta*\lambda \in \irs(G)$.
\end{claim}

Let $\Gamma$ be a finite index subgroup of $G$, and let $\theta$ be
the hitting measure on $\Gamma$ of the $\mu$ random walk on $G$. Let
$\lambda$ be a $\Gamma$ IRS measure, so that $(B(\sub{\Gamma}), \nu_\lambda)$
is a $(\Gamma,\theta)$ Bowen space. It follows that $(B(\sub{G}),
\nu_{\eta * \lambda})$ is a $(G, \mu)$ Bowen space.  Since $\Gamma$ is
finite index in $G$, every $(G,\mu)$-stationary space is also a
$(\Gamma,\theta)$-stationary space~\cite{furstenberg1971random}. In
particular, $(B(\sub{G}),\nu_{\eta * \lambda})$ is also a $(\Gamma,
\theta)$-stationary space. Furthermore, $(G,\mu)$ and
$(\Gamma,\theta)$ share the same Poisson boundary $(B,\nu)$, and so
there is no ambiguity in referring to the measure $\nu_{\eta
  *\lambda}$, when considering $(B(\sub{G}),\nu_{\eta * \lambda})$ as
either a $(G,\mu)$ Bowen space or a $(\Gamma,\theta)$ Bowen space.

Note that when $\Gamma$ is normal in $G$ then $\eta * \lambda$ is
supported on subgroups of $\Gamma$. In this case
$(B(\sub{\Gamma}),\nu_{\eta * \lambda})$ is both a $(G,\mu)$ and a
$(\Gamma,\theta)$ Bowen space. It may be the case that it is ergodic
with respect to the $G$ action, but not with respect to the $\Gamma$
action.

\subsection{The entropy of Bowen spaces lifted from finite index
  normal subgroups}
\label{sec:entropy-of-lifting}

We now return to consider discrete groups. In particular, let $G$ be a
discrete group with $\Gamma \lhd G$ a finite index normal
subgroup. Let $\mu$ be a generating measure on $G$, and let $\theta$
be the hitting measure on $\Gamma$.  Denote by $\theta_g$ the measure
on $\Gamma$ given by $\theta_g\left(\gamma\right) =
\theta(\gamma^g)$. Note that if $\theta$ has finite first moment then
so does $\theta_g$.  Note also that the entropy of a random variable
drawn from $\theta_g^n$ is independent of $g$, and so
$h_{RW}\left(\Gamma,\theta_g\right)$ is also independent of $g$. It
follows that the entropy of $\pb\left(\Gamma,\theta_g\right)$ is
independent of $g$.

Let $(B,\nu)$ be the Poisson boundary of $(\Gamma,\theta)$. It follows
from the definitions that $(B,g^{-1}\nu)$ is
$(\Gamma,\theta_g)$-stationary, and that furthermore
$h_{\theta_g}\left(B,g^{-1}\nu\right) = h_{\theta}(B,\nu) =
h_{RW}(\Gamma,\theta) = h_{RW}\left(\Gamma,\theta_g\right)$. Finally,
if $\lim_nZ_n\nu$ is a point mass, then $\lim_nZ_n^{g^{-1}}g^{-1}\nu =
\lim_ngZ_n\nu$ is also a point mass, and so $(B,g^{-1}\nu)$ is a
$\left(\Gamma,\theta_g\right)$-boundary. As a maximum entropy
boundary, it is the Poisson boundary of $\left(\Gamma,\theta_g\right)$.

Let $\lambda \in \irs(\Gamma)$, and let $(B(\sub{\Gamma}),
\nu_\lambda)$ be the associated $(\Gamma,\theta)$ Bowen space. Then
it follows from the discussion above that $(B(\sub{\Gamma}),
(g^{-1}\nu)_\lambda)$ is the associated $\left(\Gamma,\theta_g\right)$
Bowen space.

We are now ready to present the following result, which relates the
entropy of a lifted Bowen space $(B(\sub{\Gamma}), \nu_{\eta *
  \lambda})$, to the entropies of the original space, with respect to
the different conjugated measures $\theta_g$.

\begin{lemma}
  \label{lemma:lifting}
  Let $G$ be a discrete group with $\Gamma \lhd G$, $[G:\Gamma] <
  \infty$. Let $\mu$ be a generating measure on $G$, and let $\theta$
  be the hitting measure on $\Gamma$. 

  Let $(B(\sub{\Gamma}), \nu_\lambda)$ be a $\Gamma$ Bowen space. Then
  \begin{align*}
    h_\theta(B(\sub{\Gamma}), \nu_{\eta * \lambda}) = \frac{1}{[G:\Gamma]}
    \sum_{g\Gamma \in G/\Gamma}h_{\theta_g}\left(B(\sub{\Gamma}),\left(g^{-1}\nu\right)_\lambda\right)
  \end{align*}
  and
  \begin{align*}
    h_\mu(B(\sub{\Gamma}), \nu_{\eta * \lambda}) = \frac{1}{[G:\Gamma]^2}
    \sum_{g\Gamma \in G/\Gamma}h_{\theta_g}\left(B(\sub{\Gamma}),\left(g^{-1}\nu\right)_\lambda\right).
  \end{align*}
  
\end{lemma}
\begin{proof}
  $(B(\sub{\Gamma}), \nu_{\eta * \lambda})$ is both a $(G, \mu)$ and a
  $(\Gamma, \theta)$ Bowen space. Its $\theta$-entropy is given by
  Eq.~\ref{eq:bowen-ent} as
  \begin{align*}
    h_\theta(B(\sub{\Gamma}), \nu_{\eta * \lambda}) = \lim_{n \to
      \infty}\frac{1}{n}\int_{\sub{\Gamma}} \Ent{KZ_n}d(\eta * \lambda)(K),
  \end{align*}
  where $(Z_1,Z_2,\ldots)$ is here a $\theta$ random walk on $\Gamma$.
  We can now rewrite this as
  \begin{align*}
    h_\theta(B(\sub{\Gamma}), \nu_{\eta * \lambda}) &= \lim_{n \to
      \infty}\frac{1}{n}\frac{1}{[G:\Gamma]}\sum_{g\Gamma \in G / \Gamma}\int_{\sub{\Gamma}}
    \Ent{KZ_n}d(g\lambda)(K) \\
    &= \lim_{n \to
      \infty}\frac{1}{n}\frac{1}{[G:\Gamma]}\sum_{g\Gamma \in G/\Gamma}\int_{\sub{\Gamma}}
    \Ent{K^gZ_n}d\lambda(K).
  \end{align*}
  Note that
  \begin{align*}
   \Ent{K^gZ_n} = \Ent{gKg^{-1}Z_n} = \Ent{Kg^{-1}Z_ng} = \Ent{KZ_n^{g^{-1}}},
  \end{align*}
  and so
  \begin{align}
    \label{eq:ent-decompo}
    h_\theta(B(\sub{\Gamma}), \nu_{\eta * \lambda})
    &= \lim_{n \to
      \infty}\frac{1}{n}\frac{1}{[G:\Gamma]}\sum_{g\Gamma \in G/\Gamma}\int_{\sub{\Gamma}}
    \Ent{KZ_n^{g^{-1}}}d\lambda(K).
  \end{align}
  By another application of Eq.~\ref{eq:bowen-ent} we have that
  \begin{align*}
    h_{\theta_g}\left(B(\sub{\Gamma}), (g^{-1}\nu)_\lambda\right) &= \lim_{n \to
      \infty}\frac{1}{n}\int_{\sub{\Gamma}}
    \Ent{KZ_n^{g^{-1}}}d\lambda(K).
  \end{align*}
  Applying this to Eq.~\ref{eq:ent-decompo} yields
  \begin{align*}
    h_\theta(B(\sub{\Gamma}), \nu_{\eta * \lambda}) 
    &= \frac{1}{[G:\Gamma]}\sum_{g\Gamma \in G/\Gamma}h_{\theta_g}\left(B(\sub{\Gamma}),\left(g^{-1}\nu\right)_\lambda\right).
  \end{align*}
  Finally, we apply Eq.~\ref{eq:abramov}, which states that the ratio
  between the $(\Gamma,\theta)$ entropy and the $(G,\mu)$ entropy is
  $[G:\Gamma$]. This yields
  \begin{align*}
    h_\mu(B(\sub{\Gamma}), \nu_{\eta * \lambda}) 
    &= \frac{1}{[G:\Gamma]^2}\sum_{g\Gamma \in G/\Gamma}h_{\theta_g}\left(B(\sub{\Gamma}),\left(g^{-1}\nu\right)_\lambda\right).
  \end{align*}
\end{proof}

\subsection{No entropy gap for virtually free groups}
\label{sec:proof-of-main-thm}

\begin{proof}[Proof of Theorem~\ref{thm:main-realization}]
  Let $G$ be a finitely generated discrete group, and let $G$ have a
  free group of rank $n$ as a finite index subgroup. If $n=1$ then $G$
  is virtually $\Z$ and it follows that $\pb(G,\mu)$ is trivial, for
  any $\mu$. In particular, $G$ has no entropy gap.
  
  Consider then the case that $n \geq 2$.

  We claim that there exists a finite index subgroup $\Gamma$ in $G$
  that is a free group of rank $\geq 4$, and is furthermore normal in
  $G$: If $G$ has $F_2$ or $F_3$ as a finite index subgroup then it
  must also have a higher rank free group $F_n$ as a finite index
  subgroup.  Now, let $\Gamma$ be the normal core of $F_n$ in
  $G$. Then $\Gamma$ is a finite index subgroup of $G$, and, as a
  finite index subgroup of $F_n$, it is also free, and of rank $\geq
  n$. Denote by $\varphi$ a surjective homomorphism from $\Gamma$ to
  $(\Z/2\Z) \wr \Z^3$.
  
  Let $\mu$ be a generating probability measure on $G$, and let
  $\theta$ denote the hitting measure on $F_n$.  Since $\mu$ has
  finite first moment by the claim hypothesis, it follows from
  Lemma~\ref{thm:theta-is-finite-first-moment} that the conjugated
  measure $\theta_g$ has finite first moment, for any $g$.

  Let $\{\lambda_{p,m}\}_{m=1}^\infty$ be a sequence of invariant
  random subgroups of $\Gamma$, such that for any generating
  probability measure $\zeta$ on $\Gamma$ with finite first moment it
  holds that
  $\lim_mh_\zeta\left(B(\sub{\Gamma}),\nu_{\lambda_{p,m}}\right) =
  p \cdot h_\ell(\varphi \Gamma,\varphi_*\zeta) > 0$, as guaranteed by
  Proposition~\ref{lem:freegroup-uniform-realization}. Then by
  Lemma~\ref{lemma:lifting} above
  \begin{align*}
    h_\mu(B(\sub{\Gamma}), \nu_{\eta * \lambda_{p,m}}) 
    &= \frac{1}{[G:\Gamma]^2}\sum_{g\Gamma \in G/\Gamma}h_{\theta_g}\left(B(\sub{\Gamma}),\left(g^{-1}\nu\right)_{\lambda_{p,m}}\right).
  \end{align*}

  Taking the limits of both sides yields
  \begin{align*}
    \lim_{m \to \infty}h_\mu(B(\sub{\Gamma}), \nu_{\eta * \lambda_{p,m}}) 
    &= \frac{p}{[G:\Gamma]^2}\sum_{g\Gamma \in G/\Gamma}h_\ell(\varphi\Gamma,\varphi_*\theta_g).    
  \end{align*}
  Note that by considering only the addend for which $g\Gamma=\Gamma$,
  it follows that
  \begin{align*}
    \lim_{m \to \infty}h_\mu(B(\sub{\Gamma}),
    \nu_{\eta * \lambda_{p,m}}) &\geq
    \frac{p}{[G:\Gamma]^2}h_\ell(\varphi \Gamma, \varphi_*\theta),
  \end{align*}
  and in particular for $m$ large enough the entropy is strictly
  positive.
  
  On the other hand, $h_\ell(\varphi\Gamma,\varphi_*\theta_{g_i}) \leq
  h_{RW}(\Gamma,\theta_{g_i}) = h_{RW}(\Gamma,\theta)$, and so
  \begin{align*}
    \lim_{m \to \infty}h_\theta(B(\sub{\Gamma}), \nu_{\eta *
      \lambda_{p,m}}) &\leq \frac{p}{[G:\Gamma]}h_{RW}(\Gamma,
    \theta),
  \end{align*}
  For every $\eps>0$ there exists a $0 \leq p \leq 1$ such that
  $\frac{p}{[G:\Gamma]}h_{RW}(\Gamma, \theta) <
  \eps$. Therefore, for large enough $m$, we get that
  \begin{align*}
    0 < h_\mu(B(\sub{\Gamma}), \nu_{\eta * \lambda_{p,m}}) < \eps.
  \end{align*}
  Hence for each $\eps>0$ there exists an ergodic $(G, \mu)$
  stationary space with positive entropy that is less than
  $\eps$. We conclude that $(G,\mu)$ has no entropy gap for any
  finite first moment measure $\mu$.
  
\end{proof}

\appendix

\section{Long range percolations on amenable groups}
\label{sec:long-range-percolation}
\begin{proof}[Proof of Lemma~\ref{lemma:long-range-percolation}]

  Let $\{F_m\}_{m=1}^\infty$ be a F{\o}lner sequence in $\Gamma$. For
  each $0 \leq p \leq 1$ and $m \in \N$ we construct a corresponding
  percolation measure $\lambda_{p,m}$ as follows. Let
  $q=(1-p)^{1/|F_m|}$, and let $\alpha$ be the ergodic i.i.d.\
  percolation measure on $\Gamma$ with parameter $q$, so that under
  $\alpha$ each element $\gamma$ is open w.p.\ $q$. Let $\gamma$ be
  open under $\lambda_{p,m}$ if and only if all of the elements in
  $\gamma F_m$ were open under $\alpha$. Let $Q \sim \alpha$ and $R
  \sim \lambda_{p,m}$, so that $\gamma \in R$ if only if $\gamma F_m
  \subseteq Q$.
  
  $\lambda_{p,m}$ is clearly $\Gamma$-invariant. It is ergodic, since
  it is a factor of $\alpha$. Furthermore,
  \begin{align*}
    \P{\mbox{$\gamma$ is open}} = \P{\gamma \in R} = \P{\gamma F_m
      \subseteq Q} = q^{|\gamma F_m|} = q^{|F_m|} = 1-p.
  \end{align*}

  Let $S$ be a finite subset of $\Gamma$, and let $\gamma_0,\gamma \in
  S$. We would like to show that for any $\eps$ there exists an $m$
  large enough for which it holds that the probability that one is
  open and the other not is at most $\eps/|S|$. This, by the union
  bound, will establish the claim. Assume without loss of generality
  that $\gamma_0=e$.

  Since $S$ is finite, for each $\delta$, there exists $m$ large
  enough such that $|F_m \symdiff \gamma F_m| < \delta |F_m|$ for all
  $\gamma \in S$, by the definition of a F{\o}lner sequence. Choose
  $m$ large enough so that $\delta < \eps/|S|$ and also $\delta <
  \frac{\log( 1-\eps/|S|)}{\log (1-p)}$, or $(1-p)^\delta > 1-\eps/|S|$.
  
  Consider first the case that $e$ is open in $\lambda_{p,m}$. Then
  \begin{align*}
    \CondP{\gamma \in R}{e \in R} &= \CondP{\gamma F_m \subseteq Q}{F_m \subseteq
      Q} \\
    &= \CondP{F_m \cap \gamma F_m \subseteq Q, \gamma F_m \setminus F_m \subseteq
      Q}{F_m \subseteq Q}\\
    &= \CondP{\gamma F_m \setminus F_m \subseteq Q}{F_m \subseteq Q}. 
  \end{align*}
  Since $\alpha$ is i.i.d.,
  \begin{align*}
    = q^{|\gamma F_m \setminus
      F_m|} >
    (1-p)^\delta > 1-\eps/|S|.
  \end{align*}

  Consider now the case that $e$ is closed in $\lambda_{p,m}$. Then
  there exists an element $h \in F_m \setminus Q$. Hence, by
  $\Gamma$-invariance, with probability greater than $|F_m \cap \gamma
  F_m|/|F_m|$, this $h$ belongs to $(F_m \cap \gamma F_m) \setminus
  Q$, and in particular to $\gamma F_m \setminus Q$. By definition,
  this implies that $\gamma$ is also closed in $\lambda_{p,m}$. Hence
  \begin{align*}
    \CondP{\gamma \not \in R}{e \not \in R} > \frac{|F_m \cap
      \gamma F_m|}{|F_m|} > 1-\delta > 1-\eps/|S|.
  \end{align*}
\end{proof}

\section{Hitting measures and finite first moments}
\label{sec:hitting-measures-proof}
\begin{lemma*}[\ref{thm:theta-is-finite-first-moment}]
  Let $G$ be a finitely generated group, and let $\Gamma\le G$ with
  $[G:\Gamma]<\infty$. If $\mu\in\mathcal{P}(G)$ has finite first
  moment, then the hitting measure $\theta\in \mathcal{P}(\Gamma)$
  also has finite first moment.
\end{lemma*}

By definition, $\theta$ is of finite first moment if
\begin{align}
  \label{eq:theta-finite-moment}
  \sum_{\gamma\in\Gamma}\theta(\gamma)|\gamma|_\mS < \infty
\end{align}
where $\mS$ is some finite symmetric generating set of $\Gamma$. Since
finite index subgroups are quasi-isometric to the group, it is enough
to check the condition in Eq.~\ref{eq:theta-finite-moment} for a word
length metric that is induced by a word length metric of $G$, or,
equivalently, for $\mS$ a finite symmetric generating set of $G$.

Consider the $\mu$ random walk $(Z_1, Z_2, \ldots)$ on $G$.  Fix
$\mS$, a finite symmetric generating set of $G$, and denote
$|g|=|g|_\mS$. Let $L_{n}=|Z_{n}|$.  The first moment of $\mu$ can
be written as $\E{L_1}$, and so $C_1=\E{L_1}<\infty$.

Denote by $\tau$ be the $\Gamma$-hitting time of the $\mu$ random
walk, and recall (Section~\ref{sec:subgroups}) that $\theta$, the hitting
measure, is the law of $Z_\tau$. Then the first moment of the hitting
measure $\theta$ is $\E{L_\tau}$. We therefore need to show that
$\E{L_\tau}<\infty$.

Let $M_n=nC_1-L_{n}$. We want to apply the optional stopping time
theorem on $M_n$. For that we prove the following claim.
\begin{claim}
  \label{claim:submartingale}
  $M_n$ is a submartingale w.r.t.\ the filtration
  $\sigma(Z_{1},\dots,Z_{n})$, and
  \begin{align*}
   \CondE{|M_{n+1}-M_n|}{Z_1,\dots,Z_n} \le 2C_1. 
  \end{align*}
\end{claim}
\begin{proof}
  By symmetry and the triangle inequality we have that $|gh| \leq
  |g|+|h|$.  Now,
  \begin{align*}
    \CondE{L_{n+1}}{Z_1,\ldots,Z_n} &= \sum_{g \in G}\mu(g)|Z_ng|\\
    &\leq \sum_{g \in G}\mu(g)(|Z_n|+|g|) \\
    &= L_n + C_1,
  \end{align*}
  and so
  \begin{align*}
    \CondE{M_{n+1}}{Z_n} \geq M_n.
  \end{align*}
  Therefore $M_{n}$ is indeed a submartingale. To prove the bound,
  note that
  \begin{align*}
    \CondE{|M_{n+1}-M_n|}{Z_1,\ldots,Z_n} &=
    \CondE{|L_n - L_{n+1}+C_1|}{Z_1,\ldots,Z_n}\\
    &\leq C_1+\CondE{|L_{n+1}-L_n|}{Z_n}.
  \end{align*}
  By the triangle inequality and the symmetry of $\mS$ it follows that
  \begin{align*}    
    &= C_1+\sum_{g \in G}\mu(g)\left||Z_ng|-|Z_n|\right|\\
    &\leq C_1+\sum_{g \in G}\mu(g)|g|\\
    &= 2C_1.\\
  \end{align*}
\end{proof}

\begin{proof}[Proof of Lemma~\ref{thm:theta-is-finite-first-moment}]
  In general, the index of $\Gamma$ in $G$ is equal to the expected
  hitting time~\cite{hartman2012abramov}, and so $\E{\tau} =
  [G:\Gamma] < \infty$. It follows by Theorem (7.5)
  in~\cite{durrett1996probability} that because $M_n$ is a
  submartingale satisfying the condition of
  Claim~\ref{claim:submartingale}, then
  \begin{align*}
    \E{M_\tau} \geq \E{M_1}=0.
  \end{align*}
  Hence
  \begin{align*}
    \E{\tau C_1 - L_\tau} \geq 0
  \end{align*}
  and since $\E{\tau} = [G:\Gamma]$ then
  \begin{align*}
    \E{L_\tau} \leq [G:\Gamma] C_1 < \infty.
  \end{align*}
\end{proof}

\bibliography{realization}
\end{document}